\documentclass{article}

\usepackage{arxiv}
\usepackage{enumerate}
\usepackage[utf8]{inputenc} 
\usepackage[T1]{fontenc}    
\usepackage{hyperref}       
\usepackage{url}            
\usepackage{booktabs}       
\usepackage{amsthm,amssymb,amsmath,amsfonts,amscd,mathrsfs,bm}       
\usepackage{nicefrac}       
\usepackage{microtype}      
\usepackage{lipsum}
\usepackage{graphicx}
\usepackage{algorithm}
\usepackage{algpseudocode}
\usepackage[table,xcdraw]{xcolor}
\usepackage[title]{appendix}
\usepackage{mathtools}

\usepackage{graphicx}

\usepackage{array}
\newcolumntype{L}[1]{>{\raggedright\let\newline\\\arraybackslash\hspace{0pt}}m{#1}}
\newcolumntype{C}[1]{>{\centering\let\newline\\\arraybackslash\hspace{0pt}}m{#1}}

\theoremstyle{definition}
\newtheorem{definition}{Definition}[section]

\theoremstyle{plain}
\newtheorem{lemma}{Lemma}[section]
\newtheorem{theorem}{Theorem}

\theoremstyle{remark}
\newtheorem{remark}{Remark}

\makeatletter
\newcommand{\oset}[3][0ex]{%
  \mathrel{\mathop{#3}\limits^{
    \vbox to#1{\kern-2\ex@
    \hbox{$\scriptstyle#2$}\vss}}}}
\makeatother

\newfloat{procedure}{htbp}{loa}
\floatname{procedure}{Procedure} 

\title{A qualitative study of an anaerobic food-web reveals the importance of hydrogen for microbial stability}

\author{
Szymon Sobieszek \\
  Department of Mathematics \& Statistics\\
  McMaster University\\
  Hamilton, ON Canada \\
  \texttt{sobieszs@mcmaster.ca} \\
\And
 Gail S. K. Wolkowicz \\
  Department of Mathematics \& Statistics\\
  McMaster University\\
  Hamilton, ON Canada \\
  \texttt{wolkowic@mcmaster.ca} \\
  \And
    Matthew J. Wade \\
    School of Engineering\\
    Newcastle University, Newcastle-upon-Tyne \\
    United Kingdom \\
  \texttt{matthew.wade@ncl.ac.uk} \\
}

\begin{document}
\maketitle

\begin{abstract}
The mathematical analysis of a three-tiered food-web describing anaerobic chlorophenol mineralisation has suggested the emergence of interesting dynamical behaviour through its specific ecological interactions, which include competition, syntrophy and product inhibition. Previous numerical analyses have revealed the possibility for a Hopf bifurcation occurring through the interior equilibrium and the role of extraneous substrate inputs in both mitigating the emergence of periodic solutions and expanding the desired stable positive steady-state, where full mineralisation occurs. Here we show that, for a generalised model, the inflow of multiple substrates results in greater dynamical complexity and prove the occurrence of a supercritical Hopf bifurcation resulting from variations in these operating parameters. Further, using numerical estimation, we also show that variations in the dilution rate can lead to Bogdanov-Takens and Bautin bifurcations. Finally, we are able to show apply persistence theory for a range of parameter sets to demonstrate unique persistence in the cases where chlorophenol and hydrogen are extraneously added to the system, mirroring recent applied studies highlighting the role of hydrogen in maintaining stable anaerobic microbial communities.
\end{abstract}


\section{Introduction}
A mathematical model of the complete anaerobic mineralisation of a generic monochlorophenol isomer (C$_6$H$_4$ClOH) by a canonical food-web has recently been described \cite{wade2016}. The system comprises three microbial species (chlorophenol degrader, phenol degrader, hydrogenotrophic methanogen) whose interactions are summarised as follows:

\begin{itemize}
    \item Reductive dehalogenation of chlorophenol in the presence of hydrogen by the chlorophenol degrader producing phenol \cite{mazur2003};
    \item Phenol is mineralised to acetate and hydrogen via the benzoyl-CoA pathway or a caproate intermediary \cite{leven2012};
    \item Anaerobic acetogenic (acetate producing) processes are known to be endergonic  (the reaction results in a net loss of energy to the system). The production of hydrogen can lead to thermodynamic constraints, or inhibition, if its partial pressure is high enough. In other words, the reaction becomes decreasingly exergonic as more hydrogen is produced until it ceases to be thermodynamically favourable \cite{schink97,grosskopf16}. Hydrogen scavengers such as the methanogen form a syntrophic partnership with the phenol degrader by maintaining the hydrogen partial pressure  at concentrations low enough for the mineralisation reaction to proceed;
    \item Given that the chlorophenol degrader may also act as a syntrophic partner with the phenol degrader, a competitive interaction between the two hydrogen utilisers occurs. This positive and negative feedback loops reframes the ecological network from being a simple food-chain to a more complex food-web that allows for the possibility of periodicity. This additionally leads to the possibility of the system reducing to a self-sustaining two-species network in which the chlorophenol degrader acts as the syntrophic partner to the phenol degrader. 
\end{itemize}

For reference, the food-web is presented diagrammatically in \cite{wade2016} (Figure 8) and \cite{sari2017} (Figure 1). The model is a simplified representation of the system at the population level, ignoring metabolic intermediates and dead-end products such as methane, which does not contribute to the process dynamics. Acetoclastic methanogenesis, the conversion of acetate to methane, is also omitted from the model. 

Hydrogen, however, has been shown to play an important role in stability of anaerobic microbial communities through the effects of inhibition and competition \cite{bassani2015,chen2019,smith2019}. Given that external hydrogen addition will maintain the methanogen population (no washout when the methanogen growth rate is greater than the combined dilution and decay/maintenance rates) under a wider operating parameter regime (chlorophenol inflow and dilution rate) \cite{wade2016}, a global analysis of the model can provide deeper insights into the ecological role of hydrogen through its association with community stability and criticality of the Hopf bifurcation.

Here, we focus on the mathematical analysis of the model, extending the work reported in the literature. For example, an analytical approach was taken to characterise the existence and stability of the system equilibria with and without inclusion of a microbial decay term using a general representation of the species growth functions \cite{sari2017}. With no decay, local stability and the conditions giving rise to asymptotic coexistence of all three species have been shown analytically, where the possibility of periodic orbits are also not excluded \cite{elhajji2017}. However, numerical analysis has suggested the presence of a Hopf bifurcation emerging through the positive steady-state, with the concentration of influent chlorophenol as the bifurcating parameter \cite{sari2017}.

In this work, we extend the analysis of the model providing a proof of the existence, uniqueness and stability of six identified equilibria with the addition of only chlorophenol, and the more general case where all three substrates as external inputs to the system. The procedure we use allows to identify the sufficient conditions for the emergence of a Hopf bifurcation in those inflow concentration parameters. We are also able to prove that the dimensionless model is uniformly persistent, a new result for the system.

\section{The model revisited}

We first present concisely the original chemostat model using identical scaling to that given by \cite{wade2016}

\begin{align}\label{eq:system_scaled}
\begin{split}
\begin{cases}
    x_0' &= -\alpha x_0 + \mu_0\left(s_0,s_2 \right)x_0 -k_{_A}x_0, \\
    x_1' &= -\alpha x_1 + \mu_1\left(s_1,s_2 \right)x_1 -k_{_B}x_1,\\
    x_2' &= -\alpha x_2 + \mu_2\left(s_2 \right)x_2 - k_{_C}x_2,\\
    s_0' &= \alpha\left(u_f - s_0\right) - \mu_0\left(s_0,s_2\right)x_0,\\
    s_1' &= \alpha (u_g-s_1) +\omega_0\mu_0\left(s_0,s_2\right)x_0 -\mu_1 \left(s_1,s_2\right)x_1,\\
    s_2' &= \alpha (u_h-s_2) -\omega_2\mu_0\left(s_0,s_2\right)x_0 + \omega_1\mu_1\left(s_1,s_2\right)x_1 -\mu_2\left(s_2\right)x_2,
\end{cases}
\end{split}
\end{align}

with

\begin{align}
    \alpha &= \frac{D}{k_\text{m,ch}Y_\text{ch}},\\
    u_f &= \frac{S_\text{ch,in}}{K_\text{S,ch}}, \quad u_g = \frac{S_\text{ph,in}}{K_\text{S,ph}}, \quad u_h = \frac{S_\text{H$_2$,in}}{K_\text{S,H$_2$}},\\
    \omega_0 &= \frac{K_\text{S,ch}}{K_\text{S,ph}}\frac{224}{208}\left(1-Y_\text{ch}\right), \quad \omega_1 = \frac{K_\text{S,ph}}{K_\text{S,H$_2$}}\frac{32}{224}\left(1-Y_\text{ph}\right), \quad \omega_2 = \frac{16}{208}\frac{K_\text{S,ch}}{K_\text{S,H$_2$}},\\
    \phi_1 &= \frac{k_\text{m,ph}Y_\text{ph}}{k_\text{m,ch}Y_\text{ch}}, \quad \phi_2 = \frac{k_\text{m,H$_2$}Y_\text{H$_2$}}{k_\text{m,ch}Y_\text{ch}},\\
    K_P &= \frac{K_\text{S,H$_2$,c}}{K_\text{S,H$_2$}}, \quad K_I = \frac{K_\text{S,H$_2$}}{K_{I,\text{H$_2$}}},\\
    k_{_A} &= \frac{k_\text{dec,ch}}{k_\text{m,ch}Y_\text{ch}}, \quad k_{_B} = \frac{k_\text{dec,ph}}{k_\text{m,ch}Y_\text{ch}}, \quad k_{_C} = \frac{k_\text{dec,H$_2$}}{k_\text{m,ch}Y_\text{ch}}.
\end{align}

\begin{align}
\mu_0\left(s_0,s_2\right) &= \frac{s_0}{1+s_0}\frac{s_2}{K_P+s_2}, \quad \mu_1\left(s_1,s_2\right)=\frac{\phi_1 s_1}{1+s_1}\frac{1}{1+K_I s_2}, \quad \mu_2\left(s_2\right)=\frac{\phi_2 s_2}{1+s_2},\label{eq:growth_fcns}
\end{align}

where $\alpha$ is the dilution rate, $u_f, u_g, u_h$ are the chlorophenol, phenol and hydrogen inflow concentrations, respectively, and $k_A, k_B, k_C$ are the decay (or maintenance) terms. These are scaled to be dimensionless, as are the other parameters using the scaling provided by \cite{wade2016}. Briefly, $k_\mathrm{m,\cdot}$ are the specific growth rates, $K_\mathrm{S,\cdot}$ are the half-saturation coefficients, $Y_\mathrm{\cdot}$ are the substrate yield coefficients, $K_\mathrm{I,H_2}$ is the kinetic inhibition constant of hydrogen on the phenol degrader, and $k_\mathrm{dec,\cdot}$ are the unscaled decay terms. The numeric values indicate the stoichiometric coefficients given in terms of units of Chemical Oxygen Demand rather than molarity, as is common for environmental engineering models. The $\mu_n(\cdot), n = {0,1,2}$ are the species growth functions described by double Monod, Monod with product inhibition, and Monod kinetics, respectively. Subscripts $_\mathrm{ch}, _\mathrm{ph}, _\mathrm{H_2}$ relate to the chlorophenol degrader, phenol degrader, and methanogen, respectively.

For numerical bifurcation analysis given in Section \ref{subsec:bif_diag}, we consider the same parameter values as provided in the original work, as shown in Table \ref{table0}.

\begin{table}[ht]
\centering
\begin{tabular}{lll}
\hline
Parameters & Value                 \\ \hline
$\omega_0$          & 0.1854            \\
$\omega_1$          & 1656.69         \\
$\omega_2$            & 163.08         \\
$\phi_1$          & 1.8875            \\
$\phi_2$          & 3.8113       \\
$K_P$            & 0.04          \\
$K_I$    & 7.1429               \\ \hline
\end{tabular}
\caption{Parameter regimes for the system \eqref{eq:system_scaled}.}
\label{table0}
\end{table}

\section{Reduction of the model}\label{sec:reduction}
We are able to obtain many theoretical results assuming general forms of the  growth functions provided we assume the death rates of the microbial populations are insignificant compared to the dilution rate. We thus consider the following system that is identical to system \eqref{eq:system_scaled}, except that we assume $k_{i}=0$, 
$i\in\{A,B,C\}$:
\begin{align}\label{eq:system}
\begin{split}
\begin{cases}
    x_0' &= -\alpha x_0 + \mu_0\left(s_0,s_2 \right)x_0, \\
    x_1' &= -\alpha x_1 + \mu_1\left(s_1,s_2 \right)x_1,\\
    x_2' &= -\alpha x_2 + \mu_2\left(s_2 \right)x_2,\\
    s_0' &= \alpha\left(u_f - s_0\right) - \mu_0\left(s_0,s_2\right)x_0,\\
    s_1' &= \alpha (u_g-s_1) +\omega_0\mu_0\left(s_0,s_2\right)x_0 -\mu_1 \left(s_1,s_2\right)x_1,\\
    s_2' &= \alpha (u_h-s_2) -\omega_2\mu_0\left(s_0,s_2\right)x_0 + \omega_1\mu_1\left(s_1,s_2\right)x_1 -\mu_2\left(s_2\right)x_2,
\end{cases}
\end{split}
\end{align}
\[
x_i\left(0\right)\geq 0, \quad s_i\left(0\right)\geq 0, \quad i\in\left\{1,2,3\right\}.
\]
We assume that $\mu_0\left(s_0,s_2\right)$, $\mu_1\left(s_1,s_2\right)$, $\mu_2\left(s_2\right)$ are $\mathcal{C}^1$ functions that satisfy the following general conditions:
\begin{itemize}
    \item For all $s_0\geq 0$ and $s_2\geq 0$, $\mu_0\left(0,s_2\right)=0$, $\mu_0\left(s_0,0\right)=0$. As a consequence, $\partial_{s_0}\mu_0\left(s_0,0\right)=0$, $\partial_{s_2}\mu_0\left(0,s_2\right)=0$. Thus we assume that the chlorophenol degrader cannot grow in the absence of either chlorophenol or hydrogen;
    \item For all $s_0 >0$ and $s_2 >0$, $\partial_{s_0} \mu_0\left(s_0,s_2\right)>0$, $\partial_{s_2} \mu_0\left(s_0,s_2\right)>0$. Thus we assume that the chlorophenol degrader grows on both chlorophenol and hydrogen;
    \item For all $s_2\geq 0$ and $s_1\geq 0$, $\mu_1\left(0,s_2\right)=0$, $\partial_{s_2}\mu_1\left(0,s_2\right)=0$. Thus we assume that the phenol degrader cannot grow in the absence of phenol;
    \item For all $s_1>0$ and $s_2>0$, $\partial_{s_1}\mu_1\left(s_1,s_2\right)>0$, $\partial_{s_2}\mu_1\left(s_1,s_2\right)<0$. Thus we assume that the supply of phenol results in growth of the phenol degrader, and that hydrogen inhibits its growth;
    \item $\mu_2\left(0\right)=0$ and $\mu_2'\left(s_2\right)>0$ for all $s_2>0$. Thus we assume that the mathanogen cannot grow without the presence of hydrogen, and that increasing the supply of hydrogen results in faster growth of the methanogen.
\end{itemize}
We use the prototypes $\mu_0$, $\mu_1$, and $\mu_2$ defined in \eqref{eq:growth_fcns}, which satisfy these conditions, when we are able to prove results in general, and when providing numerical simulations or bifurcation diagrams.

We now prove a lemma that we will use to show global well-posedness of system \eqref{eq:system}.
\begin{lemma}\label{lemma:pos_bound}
All solutions of system \eqref{eq:system} with positive initial conditions remain positive and bounded for all positive times.

If $x_i\left(0\right)=0$, $i\in\left\{1,2,3\right\}$, then $x_i\left(t\right)=0$ for all $t\geq 0$.
\end{lemma}

\begin{proof}
Consider any solution $\Vec{\varphi}\left(t\right)$ with positive initial conditions. By existence and uniqueness theory, there cannot be a time $\Bar{t}>0$, such that $x_i\left(\Bar{t}\right)=0$ for some $i\in\left\{1,2,3\right\}$, since then $x_i\left(t\right)\equiv 0$ for all $t\in\mathbb{R}$, contradicting $x_i\left(0\right)>0$. Hence $x_i\left(t\right)>0$ for all $t\geq 0$. Also, if $x_i\left(0\right)=0$ for some $i\in\left\{1,2,3\right\}$, then there is a solution of system \eqref{eq:system} with $x_i\left(t\right)\equiv 0$ for all $t\in\mathbb{R}$. By existence and uniqueness theory, this is the only solution.

Now, consider $\Vec{\varphi}\left(t\right)$, and suppose that there is some $\Bar{t}>0$, such that $s_0\left(t\right)>0$ for $t\in\left[0,\Bar{t}\right)$, $s_0\left(\Bar{t}\right)=0$, and $s_1\left(t\right),s_2\left(t\right)\geq 0$ for $t\in\left[0,\Bar{t}\right]$. Then $s_0'\left(\Bar{t}\right)\leq 0$. However, from system \eqref{eq:system}, $s_0'\left(\Bar{t}\right)= \alpha u_f$. If $u_f>0$, then $s_0'\left(\Bar{t}\right)>0$, a contradiction. If $u_f=0$, then there is a solution of system \eqref{eq:system} with $s_0\left(t\right)\equiv 0$, which contradicts uniqueness of solutions. It follows that $s_0\left(t\right)>0$ for all $t\geq 0$.

Next, consider $\Vec{\varphi}\left(t\right)$, and suppose that there is some $\Bar{t}>0$, such that $s_2\left(t\right)>0$ for $t\in\left[0,\Bar{t}\right)$, $s_2\left(\Bar{t}\right)=0$, and $s_1\left(t\right)\geq 0$ for $t\in\left[0,\Bar{t}\right]$. Then $s_2'\left(\Bar{t}\right)\leq 0$. However, from system \eqref{eq:system}, $s_2'\left(\Bar{t}\right)=\alpha u_h + \omega_1\mu_1\left(s_1\left(\Bar{t}\right),0\right)x_1$. If $u_h>0$, or $s_1\left(\Bar{t}\right)>0$, then $s_2'\left(\Bar{t}\right)>0$, a contradiction. If both $u_h=0$, and $s_1\left(\Bar{t}\right)=0$, then $s_2'\left(\Bar{t}\right)=0$, and there is another solution with $s_2\left(t\right)\equiv 0$, which contradicts  uniqueness of solutions to initial value problems. It follows that $s_2\left(t\right)>0$ for all $t\geq 0$.

Finally, consider $\Vec{\varphi}\left(t\right)$, and suppose that there is some $\Bar{t}>0$, such that $s_1\left(t\right)>0$ for $t\in\left[0,\Bar{t}\right)$, $s_1\left(\Bar{t}\right)=0$. Then $s_1'\left(\Bar{t}\right)\leq 0$. However, from system \eqref{eq:system}, $s_1'\left(\Bar{t}\right)=\alpha u_g + \omega_0\mu_0\left(s_0\left(\Bar{t}\right),s_2\left(\Bar{t}\right)\right)x_0$, so $s_1'\left(\Bar{t}\right)>0$, a contradiction. It follows that $s_1\left(t\right)>0$ for all $t\geq 0$.

We have thus proved the positivity of solutions and move on to showing the boundedness of solutions.

By adding the first and the fourth equations of (\ref{eq:system}), we obtain
\[
x_0'+s_0' = -\alpha\left(x_0+s_0-u_f\right),
\]
hence
\[
\left(x_0 + s_0 -u_f\right)' = -\alpha\left(x_0+s_0-u_f\right),
\]
which implies that
\begin{equation}\label{eq:cons_law_1}
    x_0(t)+s_0(t) = u_f + \left(x_0(0)+s_0(0)-u_f\right)e^{-\alpha t}.
\end{equation}
Similarly, we obtain
\begin{equation}\label{eq:cons_law_2}
    x_1(t)+\omega_0 s_0(t) + s_1(t) = \omega_0 u_f +u_g + \left(x_1(0)+\omega_0 s_0(0)+s_1(0)-\omega_0 u_f - u_g\right)e^{-\alpha t},
\end{equation}
and
\begin{multline}\label{eq:cons_law_3}
    \omega_2x_0(t) + x_2(t) +\omega_0\omega_1 s_0(t) + \omega_1 s_1(t) + s_2(t) = \omega_0\omega_1u_f + \omega_1u_g + u_h+ \\ \left(\omega_0x_0(0)+x_2(0)+\omega_0\omega_1s_0(0) + \omega_1 s_1(0) + s_2(0) -\omega_0\omega_1u_f-\omega_1u_g-u_h\right)e^{-\alpha t} .
\end{multline}
Since all terms of the sums in (\ref{eq:cons_law_1}), (\ref{eq:cons_law_2}) and (\ref{eq:cons_law_3}) are positive for all positive initial conditions, the solutions of (\ref{eq:system}) are bounded. Also, taking the limit as $t \to \infty$ in equations \eqref{eq:cons_law_1}, \eqref{eq:cons_law_2}, and \eqref{eq:cons_law_3} we obtain that 
\begin{align}
    \lim_{t\to\infty}\left(x_0\left(t\right)+s_0\left(t\right)\right) &= u_f,\\
    \lim_{t\to\infty}\left(x_1\left(t\right)+\omega_0 s_0\left(t\right) + s_1\left(t\right)\right) &= \omega_0 u_f +u_g,\\
    \lim_{t\to\infty}\left(\omega_2x_0\left(t\right) + x_2\left(t\right) +\omega_0\omega_1 s_0\left(t\right) + \omega_1 s_1\left(t\right) + s_2\left(t\right)\right) &= \omega_0\omega_1u_f + \omega_1u_g + u_h.
\end{align}
Starting with any positive initial conditions, the solutions of system \eqref{eq:system} eventually satisfy
\begin{align}
    x_0 + s_0 &= u_f,\label{eq:cons_princ_1} \\
    x_1 + \omega_0s_0+s_1 &= \omega_0 u_f+u_g,\label{eq:cons_princ_2} \\
    \omega_2x_0 +x_2 + \omega_0\omega_1s_0 + \omega_1s_1 + s_2 &= \omega_0\omega_1u_f +\omega_1u_g + u_h.\label{eq:cons_princ_3}
\end{align}
\end{proof}
We call relations \eqref{eq:cons_princ_1}-\eqref{eq:cons_princ_3} "conservation principles". In other words, system (\ref{eq:system}) admits a positively invariant attracting set $\Omega \subset \mathbb{R}^6$, such that
\begin{align}\label{eq:inv_set}
\begin{split}
    \Omega = \{\left(x_0,x_1,x_2,s_0,s_1,s_2\right)\in \mathbb{R}^6: x_i,s_i &\geq 0, i=0,1,2; \\
    x_0 + s_0 &= u_f, \\
    x_1 + \omega_0s_0+s_1 &= \omega_0 u_f+u_g, \\
    \omega_2x_0 +x_2 + \omega_0\omega_1s_0 + \omega_1s_1 + s_2 &= \omega_0\omega_1u_f +\omega_1u_g + u_h\}.
\end{split}
\end{align}
Using the conservation principles we can compute $s_0$, $s_1$, and $s_2$ as functions of $x_0$, $x_1$, $x_2$
\begin{align}\label{eq:cons_laws}
\begin{split}
    s_0 &= -x_0 + u_f,\\
    s_1 &= \omega_0 x_0 -x_1 + u_g,\\
    s_2 &= -\omega_2 x_0 + \omega_1 x_1 - x_2 + u_h.
\end{split}    
\end{align}
Now, we can reduce the analysis of the original system (\ref{eq:system}) to the analysis of the following equivalent three-dimensional system on the invariant set $\Omega$
\begin{align}\label{eq:system_reduced}
    \begin{cases}
        \begin{split}
            x_0' &= -\alpha x_0 +\mu_0\left(-x_0 + u_f,-\omega_2 x_0 +\omega_1x_1 -x_2 +u_h\right)x_0,\\
            x_1' &= -\alpha x_1 +\mu_1\left(\omega_0x_0-x_1 +u_g,-\omega_2x_0+\omega_1x_1-x_2 +u_h\right)x_1,\\
            x_2' &= -\alpha x_2 + \mu_2\left(-\omega_2x_0 +\omega_1x_1 -x_2 +u_h\right)x_2.
        \end{split}
    \end{cases}
\end{align}

From now on, we will study the reduced system (\ref{eq:system_reduced}). We begin by analyzing all the possible equilibria.

\subsection{Equilibria of the reduced system and their local stability}
The equilibria are found by setting the right hand sides of  equations in (\ref{eq:system_reduced}) equal to zero. Below, we list all the possibilities obtained this way. Since equations \eqref{eq:cons_laws} give a one-to-one correspondence of the equilibria of system \eqref{eq:system_reduced} with the equilibria of system \eqref{eq:system}, we also list the corresponding steady states $\left(x_0,x_1,x_2,s_0,s_1,s_2\right)$ of the six-dimensional system in each case.

Types of equilibria of system (\ref{eq:system_reduced}):
\begin{itemize}
    \item Zero equilibrium $_{\left(000\right)}\mathcal{E}=\left(0,0,0\right)$. The corresponding equilibrium $\prescript{}{\left(000\right)}{\mathbb{E}}$ in the six-dimensional system:
    \begin{equation}
        \prescript{}{\left(000\right)}{\mathbb{E}} = \left(0,0,0,u_f,u_g,u_h\right).
    \end{equation}
    In this case, all the populations die, hence the only source for the substrates comes from the inflow rates $u_f$, $u_g$, and $u_h$.
    \item Boundary equilibria:
    \begin{itemize}
        \item $_{\left(100\right)}\mathcal{E}=\left(\prescript{}{\left(100\right)}{x_0},0,0\right)$, where $x_0=\prescript{}{\left(100\right)}{x_0}>0$ is a solution (if it exists) of
        \begin{equation}
            \mu_0(-x_0+u_f,-\omega_2 x_0+u_h) = \alpha.
        \end{equation}
        The corresponding equilibrium $\prescript{}{\left(100\right)}{\mathbb{E}}$ in the six-dimensional system:
        \begin{equation}
            \prescript{}{\left(100\right)}{\mathbb{E}} = \left(\prescript{}{\left(100\right)}{x_0},0,0,-\prescript{}{\left(100\right)}{x_0}+u_f,\omega_0 \prescript{}{\left(100\right)}{x_0}+u_g,-\omega_2 \prescript{}{\left(100\right)}{x_0}+u_h\right).
        \end{equation}
        In this case, the only microorganism surviving is the chlorophenol degrader. It consumes the chlorophenol, hence the value of $s_0$ is given as the balance between this consumption, and the supply inflow $u_f$. Since $x_0$ produces phenol this value is added to $u_g$ in the total phenol amount $s_1$. Since $x_0$ consumes hydrogen as well, the value $\omega_2 \prescript{}{\left(100\right)}{x_0}$ is subtracted from $s_2$ as well. This steady state is not desirable because of the phenol build-up in the system.
        \item $_{\left(010\right)}\mathcal{E}=\left(0,\prescript{}{\left(010\right)}{x_1},0\right)$, where $x_1=\prescript{}{\left(010\right)}{x_1}>0$ is a solution (if it exists) of
        \begin{equation}
            \mu_1(-x_1+u_g,\omega_1x_1+u_h) = \alpha.
        \end{equation}
        The corresponding equilibrium $\prescript{}{\left(010\right)}{\mathbb{E}}$ in the six-dimensional system:
        \begin{equation}
            \prescript{}{\left(010\right)}{\mathbb{E}} = \left(0,\prescript{}{\left(010\right)}{x_1},0,u_f,-\prescript{}{\left(010\right)}{x_1}+u_g,\omega_1 \prescript{}{\left(010\right)}{x_1}+u_h\right).
        \end{equation}
        In this case, only the phenol degrader survives, and hence the value of $s_1$ at the equilibrium is equal to the balance between its consumption and inflow $u_g$. Chlorophenol is not being consumed, hence its total amount equals the inflow concentration $u_f$. Hydrogen is being produced by the phenol degrader, and also its value is increased by the inflow $u_h$.
        
        \item $_{\left(001\right)}\mathcal{E}=\left(0,0,\prescript{}{\left(001\right)}{x_2}\right)$, where $x_2=\prescript{}{\left(001\right)}{x_2}>0$ is a solution (if it exists) of
        \begin{equation}
            \mu_2(-x_2+u_h) = \alpha.
        \end{equation}
        The corresponding equilibrium $\prescript{}{\left(001\right)}{\mathbb{E}}$ in the six-dimensional system:
        \begin{equation}
            \prescript{}{\left(001\right)}{\mathbb{E}} = \left(0,0,\prescript{}{\left(001\right)}{x_2},u_f,u_g,-\prescript{}{\left(001\right)}{x_2}+u_h\right).
        \end{equation}
        Here, only the methanogen is present, hence the values of chlotophenol and phenol are equal the inflow concentrations $u_f$ and $u_g$, respectively. 
        \item $_{\left(101\right)}\mathcal{E}=\left(\prescript{}{\left(101\right)}{x_0},0,-\omega_2\prescript{}{\left(101\right)}{x_0}+u_h-\mu_2^{-1}\left(\alpha\right)\right)$, 
        where $x_0=\prescript{}{\left(101\right)}{x_0}>0$ is a solution (if it exists) of
        \begin{equation}
            \mu_0(-x_0+u_f,\mu_2^{-1}(\alpha)) = \alpha.
        \end{equation}
        The corresponding equilibrium $\prescript{}{\left(101\right)}{\mathbb{E}}$ in the six-dimensional system:
        \begin{multline}
            \prescript{}{\left(101\right)}{\mathbb{E}} =\Big( \prescript{}{\left(101\right)}{x_0},0,-\omega_2\prescript{}{\left(101\right)}{x_0}+u_h-\mu_2^{-1}\left(\alpha\right),
            -\prescript{}{\left(101\right)}{x_0}+u_f,\\
            \omega_0 \prescript{}{\left(101\right)}{x_0}+u_g,\mu_2^{\left(-1\right)}\left(\alpha\right)\Big).
        \end{multline}
        In this case, both chlorophenol degrader and methanogen are present. The lack of phenol degrader results in phenol build-up. We can also observe competition for hydrogen between the phenol degrader and methanogen.
        \item $_{\left(011\right)}\mathcal{E}=\left(0,\prescript{}{\left(011\right)}{x_1},\omega_1 \prescript{}{\left(011\right)}{x_1} + u_h -\mu_2^{-1}\left(\alpha\right)\right)$, 
        where $x_1=\prescript{}{\left(011\right)}{x_1}>0$ is a solution (if it exists) of
        \begin{equation}
            \mu_1\left(-x_1+u_g,\mu_2^{-1}\left(\alpha\right)\right) = \alpha.
        \end{equation}
        The corresponding equilibrium $\prescript{}{\left(011\right)}{\mathbb{E}}$ in the six-dimensional system:
        \begin{equation}
            \prescript{}{\left(011\right)}{\mathbb{E}} = \left(0,\prescript{}{\left(011\right)}{x_1},\omega_1 \prescript{}{\left(011\right)}{x_1} + u_h -\mu_2^{-1}\left(\alpha\right),u_f,-\prescript{}{\left(011\right)}{x_1}+u_g,\mu_2^{-1}\left(\alpha\right)\right).
        \end{equation}
        This steady state represents a two-tiered food chain, with the phenol degrader and methanogen present. Hydrogen has an inhibiting effect on the phenol degrader.
        \item $_{\left(110\right)}\mathcal{E}=\left(\prescript{}{\left(110\right)}{x_0},\prescript{}{\left(110\right)}{x_1},0\right)$, where $x_0=\prescript{}{\left(110\right)}{x_0}>0$ and $x_1=\prescript{}{\left(110\right)}{x_1}>0$ are solutions of
        \begin{align}
            \begin{split}
                \mu_0\left(-x_0 + u_f,-\omega_2 x_0 +\omega_1x_1 +u_h\right) &= \alpha,\\
                \mu_1\left(\omega_0x_0-x_1 +u_g,-\omega_2x_0+\omega_1x_1 +u_h\right) &= \alpha.
            \end{split}
        \end{align}
        The corresponding equilibrium $\prescript{}{\left(110\right)}{\mathbb{E}}$ in the six-dimensional system:
        \begin{multline}
            \prescript{}{\left(110\right)}{\mathbb{E}} = \Big(\prescript{}{\left(110\right)}{x_0},\prescript{}{\left(110\right)}{x_1},0,-\prescript{}{\left(110\right)}{x_0}+u_f,
            \omega_0 \prescript{}{\left(110\right)}{x_0} -\prescript{}{\left(110\right)}{x_1} + u_g,\\ -\omega_2 \prescript{}{\left(110\right)}{x_0} + \omega_1 \prescript{}{\left(110\right)}{x_1} + u_h \Big).
        \end{multline}
        In this case, both the chlorophenol and phenol degraders are present, however the methanogen is washed out. Thus, full mineralisation to methane is not possible and, hence, the hydrogen accumulates to some theoretical maximum, balanced such that the the inhibitory effect on the phenol degrader  does not induce washout, whilst providing enough hydrogen for chlorophenol degrader activity. 
        \end{itemize}
    \item Positive (interior) equilibrium $_{\left(111\right)}\mathcal{E}=\left(\overset{*}{x}_0,\overset{*}{x}_1,\overset{*}{x}_2\right)$, where $x_0=\overset{*}{x}_0>0$, $x_1=\overset{*}{x}_1>0$, and $x_2=\overset{*}{x}_2>0$ are solutions of
    \begin{align}\label{eq:pos_eq_implicit}
        \begin{split}
            \mu_0\left(-x_0 + u_f,-\omega_2 x_0 +\omega_1x_1 -x_2 +u_h\right)&=\alpha,\\
            \mu_1\left(\omega_0x_0-x_1 +u_g,-\omega_2x_0+\omega_1x_1-x_2 +u_h\right)&=\alpha,\\
            \mu_2\left(-\omega_2x_0 +\omega_1x_1 -x_2 +u_h\right)&=\alpha.
        \end{split}
    \end{align}
    The corresponding equilibrium $\prescript{}{\left(111\right)}{\mathbb{E}}$ in the six-dimensional system:
    \begin{equation}
        \prescript{}{\left(111\right)}{\mathbb{E}} = \left(\overset{*}{x}_0,\overset{*}{x}_1,\overset{*}{x}_2,-\overset{*}{x}_0+u_f,\omega_0 \overset{*}{x}_0 -\overset{*}{x}_1 +u_g, \mu_2^{\left(-1\right)}\left(\alpha\right) \right).
    \end{equation}
    Here, all species are present, and thus we observe full chlorophenol mineralisation. For this reason, asymptotic stability of this equilibrium is the desired operational state.
\end{itemize}

As it is not clear whether the listed equations have solutions, and if the solutions are unique, we now derive conditions on the parameters that address these questions.

\subsubsection{Existence and uniqueness}\label{subsec:exist_uniq}
Since there are many parameters in system \eqref{eq:system_reduced}, it was not possible to obtain explicit expressions for some of the equilibria. We did however simplify the computation by only looking for equilibria in the invariant set $\Omega$. This assumption is reasonable, since the dynamics of the original system reduces to the dynamics on the set $\Omega$.
\begin{itemize}
    \item $_{\left(000\right)}\mathcal{E}=\left(0,0,0\right)$ equilibrium always exists.
    \item $_{\left(100\right)}\mathcal{E}=\left(\prescript{}{\left(100\right)}{x_0},0,0\right)$. As mentioned at the beginning of subsection \ref{subsec:exist_uniq}, we are looking for the equilibria in the feasible set $\Omega$, i.e., where all of the components in corresponding six-dimensional equilibria are nonnegative. Thus, we want $x_0=\prescript{}{\left(100\right)}{x_0}$ to satisfy $\prescript{}{\left(100\right)}{x_0}\in\left(0,u_f\right]$, and $\prescript{}{\left(100\right)}{x_0}\leq \frac{u_h}{\omega_2}$; hence we consider only $x_0\in\left(0,\min{\left(u_f,\frac{u_h}{\omega_2}\right)}\right]$. For such $x_0$ the differentiable mapping $x_0 \mapsto \mu_0\left(u_f-x_0,-\omega_2x_0+u_h\right)$ is decreasing, and thus $_{\left(100\right)}\mathcal{E}$ exists in $\Omega$ if and only if $\alpha\in\left[0,\mu_0(u_f,u_h)\right)$, and when it exists, it is unique.
    \item $_{\left(010\right)}\mathcal{E}=\left(0,\prescript{}{\left(010\right)}{x_1},0\right)$. By a similar argument, we consider $x_1\in \left(0,u_g\right]$. The differentiable mapping $x_1\mapsto \mu_1(u_g-x_1,\omega_1x_1+u_h)$ is decreasing, so $_{\left(010\right)}\mathcal{E}$ exists in $\Omega$ if and only if $\alpha\in\left[0,\mu_1(u_g,u_h)\right)$, and when it exists, it is unique.
    \item $_{\left(001\right)}\mathcal{E}=\left(0,0,\prescript{}{\left(001\right)}{x_2}\right)$. Once again, we consider only $x_2\in(0,u_h]$, for which the differentiable mapping $x_2\mapsto \mu_2(-x_2+u_h)$ is decreasing, so $_{\left(001\right)}\mathcal{E}$ exists in $\Omega$ if and only if $\alpha\in[0,\mu_2(u_h))$, and when it exists, it is unique (notice that if $\mu_2(s_2)=\frac{\phi_2s_2}{1+s_2}$, we have $\prescript{}{\left(001\right)}{x_2}=u_h-\frac{\alpha}{\phi_2-\alpha}$).
    \item $_{\left(101\right)}\mathcal{E}=\left(\prescript{}{\left(101\right)}{x_0},0,-\omega_2\prescript{}{\left(101\right)}{x_0}+u_h-\mu_2^{-1}\left(\alpha\right)\right)$. For $\prescript{}{\left(101\right)}{x_2}=-\omega_2\prescript{}{\left(101\right)}{x_0}+u_h-\mu_2^{-1}\left(\alpha\right)$, the restriction $\prescript{}{\left(101\right)}{x_2}>0$ gives us the condition $\prescript{}{\left(101\right)}{x_0}< \frac{u_h-\mu_2^{-1}(\alpha)}{\omega_2}$ (notice that this already implies that $\prescript{}{\left(101\right)}{x_0}\leq \frac{u_h}{\omega_2}$ is satisfied), and thus the requirement $\prescript{}{\left(101\right)}{x_0}>0$ results in the first condition on $\alpha$, i.e., $\alpha<\mu_2\left(u_h\right)$. We also require that $\prescript{}{\left(101\right)}{x_0}\leq u_f$. We therefore only consider the mapping $x_0\mapsto \mu_0(-x_0+u_f,\mu_2^{-1}(\alpha))$ for $x_0\in\left(0, \min{\left(u_f,\frac{u_h-\mu_2^{-1}\left(\alpha\right)}{\omega_2}\right)} \right)$. This differentiable mapping is decreasing, so $_{\left(101\right)}\mathcal{E}$ exists in $\Omega$ if and only if 
    \begin{equation}
    \alpha\in\left(\mu_0\left(u_f -\min{\left(u_f,\frac{u_h-\mu_2^{-1}\left(\alpha\right)}{\omega_2}\right)},\mu_2^{-1}\left(\alpha\right) \right),\mu_0\left(u_f,\mu_2^{-1}\left(\alpha\right)\right)\right),
    \end{equation}
    and $\alpha<\mu_2\left(u_h\right)$, and when it exists, it is unique.

By solving the equation
\begin{equation}
\mu_0(-x_0+u_f,\mu_2^{-1}(\alpha)) = \alpha,
\end{equation}
we obtain the following explicit formulas for $\prescript{}{\left(101\right)}{x_0}$ and $\prescript{}{\left(101\right)}{x_2}$ with our test prototypes $\mu_0$, $\mu_1$, and $\mu_2$:
\begin{align}
\prescript{}{\left(101\right)}{x_0}&=\frac{\alpha\left(1+u_f\right)\left(K_P+\mu_2^{-1}\left(\alpha\right)\right)-u_f\mu_2^{-1}\left(\alpha\right)}{\alpha\left(K_P+\mu_2^{-1}\left(\alpha\right)\right)-\mu_2^{-1}\left(\alpha\right)},\\
\prescript{}{\left(101\right)}{x_2}&=\omega_2\frac{u_f\mu_2^{-1}\left(\alpha\right)-\alpha\left(1+u_f\right)\left(K_P+\mu_2^{-1}\left(\alpha\right)\right)}{\alpha\left(K_P+\mu_2^{-1}\left(\alpha\right)\right)-\mu_2^{-1}\left(\alpha\right)}+u_h-\mu_2^{-1}\left(\alpha\right).
\end{align}
	\item $_{\left(011\right)}\mathcal{E}=\left(0,\prescript{}{\left(011\right)}{x_1},\omega_1 \prescript{}{\left(011\right)}{x_1} + u_h -\mu_2^{-1}\left(\alpha\right)\right)$. For $x_2=\omega_1\prescript{}{\left(011\right)}{x_1} + u_h -\mu_2^{-1}(\alpha)$ the restriction $\prescript{}{\left(011\right)}{x_2}> 0$ gives the condition $x_1> \frac{\mu_2^{-1}\left(\alpha\right)-u_h}{\omega_1}$, so the requirement $\prescript{}{\left(011\right)}{x_1}\leq u_g$ results in the first condition on $\alpha$, i.e., $\alpha<\mu_2\left(\omega_1u_g+u_h\right)$ (notice that this already implies that $\omega_1\prescript{}{\left(011\right)}{x_1}+u_h\geq 0 $ is satisfied). We therefore only consider the mapping $x_1\mapsto \mu_1\left(-x_1+u_g,\mu_2^{-1}(\alpha)\right)$ for $\left(\max\left(0,\frac{\mu_2^{-1}\left(\alpha\right)-u_h}{\omega_1}\right),u_g\right]$. For such $x_1$, this differentiable mapping is decreasing, so $_{\left(011\right)}\mathcal{E}$ exists in $\Omega$ if and only if $\alpha\in \left[0,\mu_1\left(u_g-\max{\left(0,\frac{\mu_2^{-1}\left(\alpha\right)-u_h}{\omega_1}\right)},\mu_2^{-1}\left(\alpha\right)\right)\right)$ and $\alpha<\mu_2\left(\omega_1u_g+u_h\right)$, and when it exists, it is unique.
By solving the equation
\begin{equation}
\mu_1\left(-x_1+u_g,\mu_2^{-1}\left(\alpha\right)\right) = \alpha,
\end{equation}
for the prototypes given by \eqref{eq:growth_fcns}, we obtain the following explicit formulas for $\prescript{}{\left(011\right)}{x_1}$ and $\prescript{}{\left(011\right)}{x_2}$ 
\begin{align}
 \prescript{}{\left(011\right)}{x_1}&=\frac{\alpha \left(1+u_g\right) \left(1+K_I\mu_2^{-1}\left(\alpha\right)\right)-u_g\phi_1}{\alpha\left(1+K_I\mu_2^{-1}\left(\alpha\right)\right)-\phi_1},\\
 \prescript{}{\left(011\right)}{x_2}&=\omega_1\frac{\alpha \left(1+u_g\right) \left(1+K_I\mu_2^{-1}\left(\alpha\right)\right)-u_g\phi_1}{\alpha\left(1+K_I\mu_2^{-1}\left(\alpha\right)\right)-\phi_1}+u_h-\mu_2^{-1}\left(\alpha\right).
\end{align}

\item $_{\left(110\right)}\mathcal{E}=\left(\prescript{}{\left(110\right)}{x_0},\prescript{}{\left(110\right)}{x_1},0\right)$. This case is much more complicated since we cannot explicitly compute $x_0$ as a function of $x_1$, or $x_1$ as a function of $x_0$ in the same way as in the previous cases. In this case more than one equilibrium of the form $_{\left(110\right)}\mathcal{E}$ can exist. In the case of the  growth functions defined in (\ref{eq:growth_fcns}), it was proved in \cite{sari2017} that if $u_g=u_h=0$, then there exist at most two equilibria of this form. By using the specific growth functions (\ref{eq:growth_fcns}), the equilibria, given as positive solutions of the following system of equations

\begin{equation}\label{eq:x0_x1_eq_sys}
    \begin{split}
        \frac{-x_0 + u_f}{1-x_0+u_f}\frac{-\omega_2x_0+\omega_1x_1+u_h}{K_P -\omega_2x_0+\omega_1x_1+u_h} &= \alpha,\\
        \phi_1\frac{\omega_0x_0-x_1+u_g}{1+\omega_0x_0-x_1+u_g}\frac{1}{1+K_I\omega_0x_0-x_1+u_g} &= \alpha,
    \end{split}
\end{equation}

must also satisfy $x_0<u_f$ and $\max\left(0,\frac{\omega_2x_0-u_h}{\omega_1}\right)<x_1<u_g+\omega_0x_0$. Notice that the first equation in (\ref{eq:x0_x1_eq_sys}) is linear in $x_1$, hence we can compute it as a function of $x_0$, and substitute this expression into the second equation of (\ref{eq:x0_x1_eq_sys}), obtaining a fourth order polynomial in $x_0$. Each zero of this polynomial, together with the corresponding value of $x_1$, which satisfies the aforementioned conditions, will constitute an equilibrium $\prescript{}{\left(110\right)}{\mathcal{E}}$ of system \eqref{eq:system_reduced}. Since the polynomial in $x_0$ is of order four, and $x_1$ is given as a function of $x_0$, we can have at most four equilibria of the form $\prescript{}{\left(110\right)}{\mathcal{E}}$.

\item $\prescript{}{\left(111\right)}{\mathcal{E}}=\left(\overset{*}{x}_0,\overset{*}{x}_1,\overset{*}{x}_2\right)$. For the interior equilibrium, we have to consider two cases, depending on the sign of $\omega_2u_f-u_h$, since the bounds on the values of $x_1$ are different in each case. We are looking for solutions of system (\ref{eq:pos_eq_implicit}), for which $x_0$, $x_1$, and $x_2$ satisfy

\begin{align}
    \begin{split}
        x_0 &\in\left(0,u_f\right],\\
        x_1 &\in\left(\max\left(0,\frac{\omega_2x_0-u_h}{\omega_1}\right),\omega_0x_0+u_g\right], \\
        x_2 &\in\left(0,-\omega_2x_0+\omega_1x_1+u_h\right],
    \end{split}
\end{align}
if $\omega_2u_f-u_h>0$, and
\begin{align}
    \begin{split}
        x_0 &\in\left(0,u_f\right],\\
        x_1 &\in\left(0,\omega_0x_0+u_g\right], \\
        x_2 &\in\left(0,-\omega_2x_0+\omega_1x_1+u_h\right],
    \end{split}
\end{align}
if $\omega_2u_f-u_h<0$. In both cases, if we let $x_2=-\omega_2x_0+\omega_1x_1+u_h-\mu_2^{-1}\left(\alpha\right)$ (which immediately gives us a necessary condition $\alpha<\sup\limits_{s_2\geq 0}\mu_2\left(\alpha\right)$), we obtain the following system for $x_0$ and $x_1$
\begin{align}
    \begin{split}
        \mu_0\left(-x_0+u_f,\mu_2^{-1}\left(\alpha\right)\right)&=\alpha,\\
        \mu_1\left(\omega_0x_0-x_1+u_g,\mu_2^{-1}\left(\alpha\right)\right)&=\alpha.
    \end{split}
\end{align}
For $x_0\in\left(0,u_f\right]$ the differentiable mapping $x_0\mapsto\mu_0\left(-x_0+u_f,\mu_2^{-1}\left(\alpha\right)\right)$ is decreasing, so $x_0=\overset{*}{x}_0$ exists if and only if $\alpha\in\left[0,\mu_0\left(u_f,\mu_2^{-1}\left(\alpha\right)\right)\right)$ and when this value exists, it is unique. Now consider
\begin{equation}
    \mu_1\left(\omega_0\overset{*}{x}_0-x_1+u_g,\mu_2^{-1}\left(\alpha\right)\right)=\alpha.
\end{equation}
We have two cases
\begin{itemize}
    \item $\omega_2u_f-u_h>0$.
    For $x_1\in\left(\max\left(0,\frac{\omega_2\overset{*}{x}_0-u_h}{\omega_1}\right),\omega_0\overset{*}{x}_0+u_g\right]$ the differentiable mapping $x_1\mapsto\mu_1\left(\omega_0\overset{*}{x}_0-x_1+u_g,\mu_2^{-1}\left(\alpha\right)\right)$ is decreasing, so $x_1=\overset{*}{x}_1$ exists if and only if $\alpha\in\left[0,\mu_1\left(\omega_0\overset{*}{x}_0-\max\left(0,\frac{\omega_2\overset{*}{x}_0-u_h}{\omega_1}\right) + u_g, \mu_2^{-1}\left(\alpha\right)\right)\right)$ and when this value exists, it is unique.
    
    \item $\omega_2u_f-u_h<0$.
    Similarly, by considering $x_1\in\left(0,\omega_0\overset{*}{x}_0+u_g\right]$ it follows that $x_1=\overset{*}{x}_1$ exists if and only if $\alpha\in\left[0,\mu_1\left(\omega_0\overset{*}{x}_0+u_g,\mu_2^{-1}\left(\alpha\right)\right)\right)$.
\end{itemize}

Having $\overset{*}{x}_0$ and $\overset{*}{x}_1$ defined, we let $\overset{*}{x}_2=-\omega_2\overset{*}{x}_0+\omega_1\overset{*}{x}_1+u_h-\mu_2^{-1}\left(\alpha\right)$. In order to have $\overset{*}{x}_2>0$ we need $\alpha<\mu_2\left(-\omega_2\overset{*}{x}_0+\omega_1\overset{*}{x}_1+u_h\right)$; hence $\prescript{}{\left(111\right)}{\mathcal{E}}$ exists in $\Omega$ if and only if

\begin{multline}
    \alpha\in \Bigg[ 0, \min\Bigg( \mu_0\left(u_f,\mu_2^{-1}\left(\alpha\right)\right),\\ \mu_1\left(\omega_0\overset{*}{x}_0-\max\left(0,\frac{\omega_2\overset{*}{x}_0-u_h}{\omega_1}\right) + u_g, \mu_2^{-1}\left(\alpha\right)\right),\\ \mu_2\left(-\omega_2\overset{*}{x}_0+\omega_1\overset{*}{x}_1+u_h\right)\Bigg) \Bigg),
\end{multline}
in the $\omega_2u_f-u_h>0$ case, and
\begin{equation}
    \alpha\in\Big[0, \min\Big( \mu_0\left(u_f,\mu_2^{-1}\left(\alpha\right)\right),\mu_1\left(\omega_0\overset{*}{x}_0+u_g,\mu_2^{-1}\left(\alpha\right)\right), \mu_2\left(-\omega_2\overset{*}{x}_0+\omega_1\overset{*}{x}_1+u_h\right)\Big)\Big),
\end{equation}
in the $\omega_2u_f-u_h<0$ case. Although the conditions on $\alpha$ are implicit and very complicated, we now know that if $\prescript{}{\left(111\right)}{\mathcal{E}}$ exists, it is unique. With the growth functions defined in (\ref{eq:growth_fcns}), we can solve the equations (\ref{eq:pos_eq_implicit}) explicitly and obtain the following formulas for the interior equilibrium
    \begin{align}
        \overset{*}{x}_0 &= 1 + u_f + \frac{1}{K_P \left(\phi_2-\alpha\right)+\alpha -1},\\
        \overset{*}{x}_1 &= \omega_0 \overset{*}{x}_0 + u_g + 1 + \frac{\phi_1}{\alpha\left(1+K_I\frac{\alpha}{\phi_2-\alpha}\right)-\phi_1},\\
        \overset{*}{x}_2 &= -\omega_2\overset{*}{x}_0+\omega_1\overset{*}{x}_1+u_h-\frac{\alpha}{\phi_2-\alpha}.
    \end{align}
\end{itemize} 

\subsubsection{Local stability results}
We now study the local stability of the equilibria by considering the eigenvalues of the Jacobian evaluated at each equilibrium. 

The Jacobian $J$ for system (\ref{eq:system_reduced}) evaluated at $\left(x_0,x_1,x_2\right)$ has the following form
\begin{equation}
J = 
\begin{bmatrix}
\mu_0-\alpha + x_0 \left( -\frac{\partial \mu_0}{\partial s_0} -\omega_2 \frac{\partial \mu_0}{\partial s_2}\right) & \omega_1 x_0 \frac{\partial \mu_0}{\partial s_2} & -x_0 \frac{\partial \mu_0}{\partial s_2} \\
x_1 \left( \omega_0 \frac{\partial \mu_1}{\partial s_1} -\omega_2 \frac{\partial \mu_1}{\partial s_2} \right) & \mu_1 - \alpha + x_1\left(-\frac{\partial \mu_1}{\partial s_1} + \omega_1 \frac{\partial\mu_1}{\partial s_2}\right) & -x_1\frac{\partial \mu_1}{\partial s_2} \\ -\omega_2 x_2 \mu_2' & \omega_1 x_2 \mu_2' & \mu_2 -\alpha -x_2 \mu_2' 
\end{bmatrix}.
\end{equation}
\begin{itemize}
\item For the zero equilibrium $\prescript{}{\left(000\right)}{\mathcal{E}}$, the corresponding Jacobian $\prescript{}{\left(000\right)}{J}$ has the following form
\begin{equation}
\prescript{}{\left(000\right)}{J} =
\begin{bmatrix}
\mu_0 - \alpha & 0 & 0\\
0 & \mu_1-\alpha & 0\\
0 & 0 & \mu_2-\alpha
\end{bmatrix},
\end{equation}
and its eigenvalues are $\lambda_1=\mu_0\left(u_f,u_h\right) - \alpha$, $\lambda_2=\mu_1\left(u_g,u_h\right)-\alpha$, $\lambda_3=\mu_2\left(u_h\right)-\alpha$. This implies that if
\begin{itemize}
\item $\alpha>\max\left(\mu_0\left(u_f,u_h\right),\mu_1\left(u_g,u_h\right),\mu_2\left(u_h\right)\right)$, then $\prescript{}{\left(000\right)}{\mathcal{E}}$ is a stable node,
\item $\min\left(\mu_0\left(u_f,u_h\right),\mu_1\left(u_g,u_h\right),\mu_2\left(u_h\right)\right)<\alpha<\max\left(\mu_0\left(u_f,u_h\right),\mu_1\left(u_g,u_h\right),\mu_2\left(u_h\right)\right)$, then $\prescript{}{\left(000\right)}{\mathcal{E}}$ is a saddle point,
\item $\alpha<\min\left(\mu_0\left(u_f,u_h\right),\mu_1\left(u_g,u_h\right),\mu_2\left(u_h\right)\right)$, then $\prescript{}{\left(000\right)}{\mathcal{E}}$ is an unstable node.
\end{itemize}
\item For the boundary equilibrium $\prescript{}{\left(100\right)}{\mathcal{E}}$, the corresponding Jacobian $\prescript{}{\left(100\right)}{J}$ has the following form
\begin{equation}
\prescript{}{\left(100\right)}{J} = 
\begin{bmatrix}
\prescript{}{\left(100\right)}{x_0} \left( -\frac{\partial \mu_0}{\partial s_0} -\omega_2 \frac{\partial \mu_0}{\partial s_2}\right) & \omega_1 \prescript{}{\left(100\right)}{x_0} \frac{\partial \mu_0}{\partial s_2} & -\prescript{}{\left(100\right)}{x_0} \frac{\partial \mu_0}{\partial s_2} \\
0 & \mu_1 - \alpha & 0 \\
0 & 0 & \mu_2-\alpha
\end{bmatrix},
\end{equation}
and its eigenvalues are $\lambda_1=\prescript{}{\left(100\right)}{x_0} \left( -\frac{\partial \mu_0}{\partial s_0} -\omega_2 \frac{\partial \mu_0}{\partial s_2}\right)<0$,\\ $ \lambda_2=\mu_1\left(\omega_0\prescript{}{\left(100\right)}{x_0}+u_g,-\omega_2\prescript{}{\left(100\right)}{x_0}+u_h\right)-\alpha$, and $\lambda_3=\mu_2\left(-\omega_2\prescript{}{\left(100\right)}{x_0}+u_h\right)-\alpha$. Hence if
\begin{itemize}
\item $\alpha>\max\left(\mu_1\left(\omega_0\prescript{}{\left(100\right)}{x_0}+u_g,-\omega_2\prescript{}{\left(100\right)}{x_0}+u_h\right),\mu_2\left(-\omega_2\prescript{}{\left(100\right)}{x_0}+u_h\right)\right)$, then $\prescript{}{\left(100\right)}{\mathcal{E}}$ is a stable node,
\item $\alpha<\max\left(\mu_1\left(\omega_0\prescript{}{\left(100\right)}{x_0}+u_g,-\omega_2\prescript{}{\left(100\right)}{x_0}+u_h\right),\mu_2\left(-\omega_2\prescript{}{\left(100\right)}{x_0}+u_h\right)\right)$, then $\prescript{}{\left(100\right)}{\mathcal{E}}$ is a saddle point.
\end{itemize}
\item For the boundary equilibrium $\prescript{}{\left(010\right)}{\mathcal{E}}$, the corresponding Jacobian $\prescript{}{\left(010\right)}{J}$ has the following form
\begin{equation}
\prescript{}{\left(010\right)}{J} =
\begin{bmatrix}
\mu_0-\alpha & 0 & 0\\
\prescript{}{\left(010\right)}{x_1}\left( \omega_0 \frac{\partial \mu_1}{\partial s_1} -\omega_2 \frac{\partial \mu_1}{\partial s_2} \right) & \prescript{}{\left(010\right)}{x_1}\left(-\frac{\partial \mu_1}{\partial s_1} + \omega_1 \frac{\partial\mu_1}{\partial s_2}\right) & -\prescript{}{\left(010\right)}{x_1}\frac{\partial \mu_1}{\partial s_2} \\
0 & 0 & \mu_2-\alpha
\end{bmatrix},
\end{equation}
and its eigenvalues are $\lambda_1=\mu_0\left(u_f,\omega_1\prescript{}{\left(010\right)}{x_1}+u_h\right)-\alpha$, $\lambda_2=\prescript{}{\left(010\right)}{x_1}\left(-\frac{\partial \mu_1}{\partial s_1} + \omega_1 \frac{\partial\mu_1}{\partial s_2}\right)<0$, and $\lambda_3=\mu_2\left(\omega_1\prescript{}{\left(010\right)}{x_1}+u_h\right)-\alpha$. Hence if
\begin{itemize}
\item $\alpha>\max\left(\mu_0\left(u_f,\omega_1\prescript{}{\left(010\right)}{x_1}+u_h\right),\mu_2\left(\omega_1\prescript{}{\left(010\right)}{x_1}+u_h\right)\right)$, then $\prescript{}{\left(010\right)}{\mathcal{E}}$ is a stable node,
\item $\alpha<\max\left(\mu_0\left(u_f,\omega_1\prescript{}{\left(010\right)}{x_1}+u_h\right),\mu_2\left(\omega_1\prescript{}{\left(010\right)}{x_1}+u_h\right)\right)$, then $\prescript{}{\left(010\right)}{\mathcal{E}}$ is a saddle point.
\end{itemize}
\item For the boundary equilibrium $\prescript{}{\left(001\right)}{\mathcal{E}}$, the corresponding Jacobian $\prescript{}{\left(001\right)}{J}$ has the following form
\begin{equation}
\prescript{}{\left(001\right)}{J} =
\begin{bmatrix}
\mu_0-\alpha & 0 & 0\\
0 & \mu_1-\alpha & 0 \\
-\omega_2 \prescript{}{\left(001\right)}{x_2} \mu_2' & \omega_1 \prescript{}{\left(001\right)}{x_2} \mu_2' & \mu_2 -\alpha -\prescript{}{\left(001\right)}{x_2} \mu_2' 
\end{bmatrix},
\end{equation}
and its eigenvalues are $\lambda_1=\mu_0\left(u_f,-\prescript{}{\left(001\right)}{x_2}+u_h\right)-\alpha$, $\lambda_2=\mu_1\left(u_g,-\prescript{}{\left(001\right)}{x_2}+u_h\right)-\alpha$, and $\lambda_3=-\prescript{}{\left(001\right)}{x_2}\mu_2'\left(-\prescript{}{\left(001\right)}{x_2}+u_h\right)<0$. Hence if
\begin{itemize}
\item $\alpha>\max\left(\mu_0\left(u_f,-\prescript{}{\left(001\right)}{x_2}+u_h\right),\mu_1\left(u_g,-\prescript{}{\left(001\right)}{x_2}+u_h\right)\right)$, then $\prescript{}{\left(001\right)}{\mathcal{E}}$ is a stable node,
\item $\alpha<\max\left(\mu_0\left(u_f,-\prescript{}{\left(001\right)}{x_2}+u_h\right),\mu_1\left(u_g,-\prescript{}{\left(001\right)}{x_2}+u_h\right)\right)$, then $\prescript{}{\left(001\right)}{\mathcal{E}}$ is a saddle point.
\end{itemize}
\item For the boundary equilibrium $\prescript{}{\left(101\right)}{\mathcal{E}}$, the corresponding Jacobian $\prescript{}{\left(101\right)}{J}$ has the following form
\begin{equation}
\prescript{}{\left(101\right)}{J} =
\begin{bmatrix}
\prescript{}{\left(101\right)}{x_0} \left( -\frac{\partial \mu_0}{\partial s_0} -\omega_2 \frac{\partial \mu_0}{\partial s_2}\right) & \omega_1 \prescript{}{\left(101\right)}{x_0} \frac{\partial \mu_0}{\partial s_2} & -\prescript{}{\left(101\right)}{x_0} \frac{\partial \mu_0}{\partial s_2} \\
0 & \mu_1 - \alpha & 0 \\ 
-\omega_2 \prescript{}{\left(101\right)}{x_2} \mu_2' & \omega_1 \prescript{}{\left(101\right)}{x_2} \mu_2' & -\prescript{}{\left(101\right)}{x_2} \mu_2' 
\end{bmatrix},
\end{equation}
where $\prescript{}{\left(101\right)}{x_2}=-\omega_2\prescript{}{\left(101\right)}{x_0}+u_h-\mu_2^{-1}\left(\alpha\right)$. We immediately obtain one of the eigenvalues $\lambda_1=\mu_1\left(\omega_0\prescript{}{\left(101\right)}{x_0}+u_g,\mu_2^{-1}\left(\alpha\right)\right)-\alpha$. The other two eigenvalues are given as the solutions of the following quadratic equation
\begin{equation}\label{eq:char_E4}
\lambda^2 + a_1\lambda + a_0=0,
\end{equation}
where
\begin{align}
a_1 =& \prescript{}{\left(101\right)}{x_0}\left(\frac{\partial \mu_0}{\partial s_0}+\omega_2\frac{\partial \mu_0}{\partial s_2}\right) + \prescript{}{\left(101\right)}{x_2}\mu_2'\\
a_0 =& \prescript{}{\left(101\right)}{x_0}\prescript{}{\left(101\right)}{x_2}\frac{\partial \mu_0}{\partial s_0}\mu_2'.
\end{align}
Since both $a_1>0$, and $a_0>0$, by the Routh-Hurwitz criterion, all roots of the equation (\ref{eq:char_E4}) have negative real parts. Hence if
\begin{itemize}
\item $\alpha > \mu_1\left(\omega_0 \prescript{}{\left(101\right)}{x_0}+u_g,\mu_2^{-1}\left(\alpha\right)\right)$, then $\prescript{}{\left(101\right)}{\mathcal{E}}$ is asymptotically stable,
\item $\alpha < \mu_1\left(\omega_0 \prescript{}{\left(101\right)}{x_0}+u_g,\mu_2^{-1}\left(\alpha\right)\right)$, then $\prescript{}{\left(101\right)}{\mathcal{E}}$ is a saddle point.
\end{itemize}

\item For the boundary equilibrium $\prescript{}{\left(011\right)}{\mathcal{E}}$, the corresponding Jacobian $\prescript{}{\left(011\right)}{J}$ has the following form
\begin{equation}
\prescript{}{\left(011\right)}{J} = 
\begin{bmatrix}
\mu_0-\alpha & 0 & 0 \\
\prescript{}{\left(011\right)}{x_1} \left( \omega_0 \frac{\partial \mu_1}{\partial s_1} -\omega_2 \frac{\partial \mu_1}{\partial s_2} \right) &  \prescript{}{\left(011\right)}{x_1}\left(-\frac{\partial \mu_1}{\partial s_1} + \omega_1 \frac{\partial\mu_1}{\partial s_2}\right) & -\prescript{}{\left(011\right)}{x_1}\frac{\partial \mu_1}{\partial s_2} \\ -\omega_2 \prescript{}{\left(011\right)}{x_2} \mu_2' & \omega_1 \prescript{}{\left(011\right)}{x_2} \mu_2' & -\prescript{}{\left(011\right)}{x_2} \mu_2' 
\end{bmatrix},
\end{equation}
where $\prescript{}{\left(011\right)}{x_2}=\omega_1\prescript{}{\left(011\right)}{x_1}+u_h-\mu_2^{-1}(\alpha)$. We immediately obtain one of the eigenvalues $\lambda_1=\mu_0\left(u_f,\mu_2^{-1}\left(\alpha\right)\right)-\alpha$. The other two eigenvalues are given as the solutions of the following quadratic equation
\begin{equation}\label{eq:char_E5}
\lambda^2+a_1\lambda+a_0=0,
\end{equation}
where
\begin{align}
a_1 =& \prescript{}{\left(011\right)}{x_1}\left( \frac{\partial \mu_1}{\partial s_1}-\omega_1\frac{\partial \mu_1}{\partial s_2} \right)+\prescript{}{\left(011\right)}{x_2}\mu_2',\\
a_0 =& \prescript{}{\left(011\right)}{x_1}\prescript{}{\left(011\right)}{x_2}\frac{\partial \mu_1}{\partial s_1}\mu_2'.
\end{align}
Since both $a_1>0$, and $a_0>0$, by the Routh-Hurwitz criterion, all roots of the equation \eqref{eq:char_E5} have negative real parts. Hence if
\begin{itemize}
\item $\alpha > \mu_0\left(u_f,\mu_2^{-1}\left(\alpha\right)\right)$, then $\prescript{}{\left(011\right)}{\mathcal{E}}$ is asymptotically stable,
\item $\alpha < \mu_0\left(u_f,\mu_2^{-1}\left(\alpha\right)\right)$, then $\prescript{}{\left(011\right)}{\mathcal{E}}$ is a saddle point.
\end{itemize}

\item For the boundary equilibrium $\prescript{}{\left(110\right)}{\mathcal{E}}$, the corresponding Jacobian $\prescript{}{\left(110\right)}{J}$ has the following form
\begin{equation}
\prescript{}{\left(110\right)}{J} =
\begin{bmatrix}
\prescript{}{\left(110\right)}{x_0} \left( -\frac{\partial \mu_0}{\partial s_0} -\omega_2 \frac{\partial \mu_0}{\partial s_2}\right) & \omega_1 \prescript{}{\left(110\right)}{x_0} \frac{\partial \mu_0}{\partial s_2} & -\prescript{}{\left(110\right)}{x_0} \frac{\partial \mu_0}{\partial s_2} \\
\prescript{}{\left(110\right)}{x_1} \left( \omega_0 \frac{\partial \mu_1}{\partial s_1} -\omega_2 \frac{\partial \mu_1}{\partial s_2} \right) &  \prescript{}{\left(110\right)}{x_1}\left(-\frac{\partial \mu_1}{\partial s_1} + \omega_1 \frac{\partial\mu_1}{\partial s_2}\right) & -\prescript{}{\left(110\right)}{x_1}\frac{\partial \mu_1}{\partial s_2} \\ 0 & 0 & \mu_2 -\alpha
\end{bmatrix}.
\end{equation}
We immediately obtain one eigenvalue $\lambda_1=\mu_2-\alpha$. The other two eigenvalues are solutions of the following quadratic equation
\begin{equation}
\lambda^2 + a_1\lambda + a_0 = 0,
\end{equation}
where
\begin{align}
a_1 &= \prescript{}{\left(110\right)}{x_0} \left( \frac{\partial \mu_0}{\partial s_0} +\omega_2 \frac{\partial \mu_0}{\partial s_2}\right) + \prescript{}{\left(110\right)}{x_1}\left(\frac{\partial \mu_1}{\partial s_1} - \omega_1 \frac{\partial\mu_1}{\partial s_2}\right),\\
a_0 &= \prescript{}{\left(110\right)}{x_0}\prescript{}{\left(110\right)}{x_1} \left( \frac{\partial \mu_0}{\partial s_0} +\omega_2 \frac{\partial \mu_0}{\partial s_2}\right)\left(\frac{\partial \mu_1}{\partial s_1} - \omega_1 \frac{\partial\mu_1}{\partial s_2}\right) \\
&\phantom{{}=} - \omega_1 \prescript{}{\left(110\right)}{x_0}\prescript{}{\left(110\right)}{x_1}\frac{\partial \mu_0}{\partial s_2}\left(\omega_0\frac{\partial \mu_1}{\partial s_1}-\omega_2\frac{\partial \mu_1}{\partial s_2}\right). \notag
\end{align}
We have $a_1>0$, and 
\begin{equation}
a_0>0 \iff \frac{\partial \mu_0}{\partial s_0}\frac{\partial \mu_1}{\partial s_1}-\omega_1\frac{\partial \mu_0}{\partial s_0}\frac{\partial \mu_1}{\partial s_2} + \left(\omega_2 -\omega_0\omega_1\right)\frac{\partial \mu_0}{\partial s_2}\frac{\partial \mu_1}{\partial s_1} > 0.
\end{equation}
Hence if
\begin{itemize}
    \item $\alpha>\mu_2$ and $\frac{\partial \mu_0}{\partial s_0}\frac{\partial \mu_1}{\partial s_1}-\omega_1\frac{\partial \mu_0}{\partial s_0}\frac{\partial \mu_1}{\partial s_2} + \left(\omega_2 -\omega_0\omega_1\right)\frac{\partial \mu_0}{\partial s_2}\frac{\partial \mu_1}{\partial s_1}>0$ (where all the functions are evaluated at the steady state), then $\prescript{}{\left(110\right)}{\mathcal{E}}$ is asymptotically stable,
    \item $\alpha<\mu_2$ or $\frac{\partial \mu_0}{\partial s_0}\frac{\partial \mu_1}{\partial s_1}-\omega_1\frac{\partial \mu_0}{\partial s_0}\frac{\partial \mu_1}{\partial s_2} + \left(\omega_2 -\omega_0\omega_1\right)\frac{\partial \mu_0}{\partial s_2}\frac{\partial \mu_1}{\partial s_1}<0$, then $\prescript{}{\left(110\right)}{\mathcal{E}}$ is unstable.
\end{itemize}

\item For the interior equilibrium $\prescript{}{\left(111\right)}{\mathcal{E}}$, the corresponding Jacobian $\prescript{}{\left(111\right)}{J}$ has the following form
\begin{equation}
\prescript{}{\left(111\right)}{J} =
\begin{bmatrix}
 \overset{*}{x}_0 \left( -\frac{\partial \mu_0}{\partial s_0} -\omega_2 \frac{\partial \mu_0}{\partial s_2}\right) & \omega_1 \overset{*}{x}_0 \frac{\partial \mu_0}{\partial s_2} & -\overset{*}{x}_0 \frac{\partial \mu_0}{\partial s_2} \\
\overset{*}{x}_1 \left( \omega_0 \frac{\partial \mu_1}{\partial s_1} -\omega_2 \frac{\partial \mu_1}{\partial s_2} \right) & \overset{*}{x}_1\left(-\frac{\partial \mu_1}{\partial s_1} + \omega_1 \frac{\partial\mu_1}{\partial s_2}\right) & -\overset{*}{x}_1\frac{\partial \mu_1}{\partial s_2} \\ -\omega_2 \overset{*}{x}_2 \mu_2' & \omega_1 \overset{*}{x}_2 \mu_2' & -\overset{*}{x}_2 \mu_2' 
\end{bmatrix},
\end{equation}
\end{itemize}
and its eigenvalues are solutions of the following cubic equation
\begin{equation}
    \lambda^3 + a_2\lambda^2 + a_1\lambda + a_0 = 0,
\end{equation}
where
\begin{align}
    a_2 &= -\overset{*}{x}_0 \left( -\frac{\partial \mu_0}{\partial s_0} -\omega_2 \frac{\partial \mu_0}{\partial s_2}\right) - \overset{*}{x}_1\left(-\frac{\partial \mu_1}{\partial s_1} + \omega_1 \frac{\partial\mu_1}{\partial s_2}\right) +\overset{*}{x}_2 \mu_2',\\
    a_1 &= \overset{*}{x}_1 \frac{\partial \mu_1}{\partial s_1}\left(\overset{*}{x}_0\frac{\partial\mu_0}{\partial s_0}-\left(\omega_0\omega_1-\omega_2\right)\overset{*}{x}_0\frac{\partial \mu_0}{\partial s_2}+\overset{*}{x}_2\mu_2' \right)+\overset{*}{x}_0\frac{\partial \mu_0}{\partial s_0}\left(-\omega_1\overset{*}{x}_1\frac{\partial\mu_1}{\partial s_2}+\overset{*}{x}_2\mu_2'\right) ,\\
    a_0 &= \overset{*}{x}_0\overset{*}{x}_1\overset{*}{x}_2\frac{\partial \mu_0}{\partial s_0}\frac{\partial \mu_1}{\partial s_1}\mu_2'.
\end{align}
By the Routh-Hurwitz criterion, all eigenvalues have negative real parts if and only if $a_2>0$, $a_0>0$, and $a_2a_1>a_0$. We have $a_2>0$, and $a_0>0$ always, and
\begin{equation}\label{eq:int_eq_stability}
\begin{split}
    a_2a_1>a_0 \iff \left[-\overset{*}{x}_0 \left( -\frac{\partial \mu_0}{\partial s_0} -\omega_2 \frac{\partial \mu_0}{\partial s_2}\right) - \overset{*}{x}_1\left(-\frac{\partial \mu_1}{\partial s_1} + \omega_1 \frac{\partial\mu_1}{\partial s_2}\right) +\overset{*}{x}_2 \mu_2'\right]\cdot \\
    \left[\overset{*}{x}_1 \frac{\partial \mu_1}{\partial s_1}\left(\overset{*}{x}_0\frac{\partial\mu_0}{\partial s_0}-\left(\omega_0\omega_1-\omega_2\right)\overset{*}{x}_0\frac{\partial \mu_0}{\partial s_2}+\overset{*}{x}_2\mu_2' \right)+\overset{*}{x}_0\frac{\partial \mu_0}{\partial s_0}\left(-\omega_1\overset{*}{x}_1\frac{\partial\mu_1}{\partial s_2}+\overset{*}{x}_2\mu_2'\right)\right] \\ -\overset{*}{x}_0\overset{*}{x}_1\overset{*}{x}_2\frac{\partial \mu_0}{\partial s_0}\frac{\partial \mu_1}{\partial s_1}\mu_2'>0.
\end{split}
\end{equation}

Thus $\prescript{}{\left(111\right)}{\mathcal{E}}$ is asymptotically stable if $a_2a_1>a_0$, and is unstable if $a_2a_1<a_0$.

We gather the results concerning the existence and local stability of the equilibria of \eqref{eq:system_reduced} in table \ref{table:equilibria}. All the functions in the "Local stability" column are evaluated at the corresponding steady states, and all the symbols are given by equations \eqref{eq:101_star_table}-\eqref{eq:111_a0_table}. The conditions for existence of the $\prescript{}{\left(110\right)}{\mathcal{E}}$ equilibrium are discussed in subsection \ref{subsec:exist_uniq}.

\begin{table}[H]
\centering
    \begin{tabular}{C{3cm} C{5cm} C{5cm}}
        \hline
        Equilibrium  & Existence & Local stability \\
        \hline
        $_{\left(000\right)}\mathcal{E}$ & Always & $\alpha>\max(\mu_0,\mu_1,\mu_2)$ \\
        
        $ $ & $ $ & $ $\\
        
        $_{\left(100\right)}\mathcal{E}$ & $\alpha < \mu_0\left(u_f,u_h\right)$ & $\alpha>\max\left(\mu_1,\mu_2\right)$ \\
        
        $ $ & $ $ & $ $ \\
        
        $_{\left(010\right)}\mathcal{E}$ & $\alpha < \mu_1\left(u_g,u_h\right)$ & $\alpha>\max\left(\mu_0,\mu_2\right)$\\
        
        $ $ & $ $ & $ $ \\
        
        $\prescript{}{\left(001\right)}{\mathcal{E}}$ & $\alpha<\mu_2\left(u_h\right)$ & $\alpha>\max\left(\mu_0,\mu_1\right)$\\
        
        $ $ & $ $ & $ $ \\
        
        $\prescript{}{\left(101\right)}{\mathcal{E}}$ & $\prescript{}{\left(101\right)}{\star}$ & $\alpha > \mu_1$\\
        
        $ $ & $ $ & $ $ \\
        
        $\prescript{}{\left(011\right)}{\mathcal{E}}$ &
        $\prescript{}{\left(011\right)}{\star}$ & $\alpha > \mu_0$\\
        
        $ $ & $ $ & $ $ \\
        
        $\prescript{}{\left(110\right)}{\mathcal{E}}$ & - & $\alpha>\mu_2$ \text{ and } $\prescript{}{\left(110\right)}{a_0}>0$\\
        
        $ $ & $ $ & $ $ \\
        
        $\prescript{}{\left(111\right)}{\mathcal{E}}$ & $\prescript{}{\left(111\right)}{\star}$ & $\prescript{}{\left(111\right)}{a_2}\prescript{}{\left(111\right)}{a_1}-\prescript{}{\left(111\right)}{a_0}>0$,\\
        $ $ & $ $ & $ $ \\
    \end{tabular}
    \caption{Equilibria of system \eqref{eq:system_reduced} together with their local stability.}
    \label{table:equilibria}
\end{table}

\begin{align}
    \prescript{}{\left(101\right)}{\star} &= \alpha\in\left(\mu_0\left(u_f -\min{\left(u_f,\frac{u_h-\mu_2^{-1}\left(\alpha\right)}{\omega_2}\right)},\mu_2^{-1}\left(\alpha\right) \right),\mu_0\left(u_f,\mu_2^{-1}\left(\alpha\right)\right)\right),\label{eq:101_star_table} \\
    \prescript{}{\left(011\right)}{\star} &= \alpha\in \left[0,\mu_1\left(u_g-\max{\left(0,\frac{\mu_2^{-1}\left(\alpha\right)-u_h}{\omega_1}\right)},\mu_2^{-1}\left(\alpha\right)\right)\right),\\
    &\phantom{{}=} \text{ and } \alpha<\mu_2\left(\omega_1u_g+u_h\right), \notag \\
    \prescript{}{\left(110\right)}{a_0}&= \frac{\partial \mu_0}{\partial s_0}\frac{\partial \mu_1}{\partial s_1}-\omega_1\frac{\partial \mu_0}{\partial s_0}\frac{\partial \mu_1}{\partial s_2} + \left(\omega_2 -\omega_0\omega_1\right)\frac{\partial \mu_0}{\partial s_2}\frac{\partial \mu_1}{\partial s_1},\\
    \prescript{}{\left(111\right)}{\star} &= \alpha\in \Bigg[ 0, \min\Bigg( \mu_0\left(u_f,\mu_2^{-1}\left(\alpha\right)\right), \mu_1\Bigg(\omega_0\overset{*}{x}_0-\max\left(0,\frac{\omega_2\overset{*}{x}_0-u_h}{\omega_1}\right) \\ &\phantom{{}=} + u_g,\mu_2^{-1}\left(\alpha\right)\Bigg), \mu_2\left(-\omega_2\overset{*}{x}_0+\omega_1\overset{*}{x}_1+u_h\right)\Bigg) \Bigg), \notag \\
    \prescript{}{\left(111\right)}{a_2} &= -\overset{*}{x}_0 \left( -\frac{\partial \mu_0}{\partial s_0} -\omega_2 \frac{\partial \mu_0}{\partial s_2}\right) - \overset{*}{x}_1\left(-\frac{\partial \mu_1}{\partial s_1} + \omega_1 \frac{\partial\mu_1}{\partial s_2}\right) +\overset{*}{x}_2 \mu_2',\\
    \prescript{}{\left(111\right)}{a_1} &= \overset{*}{x}_1 \frac{\partial \mu_1}{\partial s_1}\left(\overset{*}{x}_0\frac{\partial\mu_0}{\partial s_0}-\left(\omega_0\omega_1-\omega_2\right)\overset{*}{x}_0\frac{\partial \mu_0}{\partial s_2}+\overset{*}{x}_2\mu_2' \right)\\
    &\phantom{{}=}+\overset{*}{x}_0\frac{\partial \mu_0}{\partial s_0}\left(-\omega_1\overset{*}{x}_1\frac{\partial\mu_1}{\partial s_2}+\overset{*}{x}_2\mu_2'\right),\notag\\
    \prescript{}{\left(111\right)}{a_0} &= \overset{*}{x}_0\overset{*}{x}_1\overset{*}{x}_2\frac{\partial \mu_0}{\partial s_0}\frac{\partial \mu_1}{\partial s_1}\mu_2'.\label{eq:111_a0_table}
\end{align}

\section{Analysis of the full system}

\subsection{Periodic orbits on the faces}\label{subsec:faces}
We begin by ruling out possibility of having a periodic orbit in one of the invariant faces of $\Omega$. This can be done by using general forms of prototypes $\mu_0$, $\mu_1$ and $\mu_2$.

First, consider system \eqref{eq:system_reduced} on the part of $\Omega$ with $x_0=0$, i.e.,

\begin{align}\label{x1_x2_system}
    \begin{cases}
        \begin{split}
            x_1' &= \left(-\alpha +\mu_1\left(-x_1 +u_g,\omega_1x_1-x_2 +u_h\right)\right)x_1,\\
            x_2' &= \left(-\alpha + \mu_2\left(\omega_1x_1 -x_2 +u_h\right)\right)x_2.
        \end{split}
    \end{cases}
\end{align}

The domain for system \eqref{x1_x2_system} is given by the following set $\Omega_{12}$:
\begin{equation}
    \Omega_{12}=\left\{ \left(x_1,x_2\right)\in \mathbb{R}^2: 0\leq x_1\leq u_g, 0\leq x_2\leq u_h+\omega_1x_1 \right\}.
\end{equation}
Notice that no periodic orbit can intersect the axes $x_1=0$ or $x_2=0$ since they are invariant. Now, let us define an auxiliary function
\begin{equation}
    \varphi_0 \left(x_1,x_2\right) = \frac{1}{x_1x_2}, \quad \left(x_1,x_2\right)\in \Omega_{12}\setminus \left( \left\{ x_1=0 \right\}\cup\left\{x_2=0\right\} \right).
\end{equation}
Then
\begin{equation}
    \nabla \cdot \left(\varphi_0\left(-\alpha +\mu_1\right)x_1 , \varphi_0\left(-\alpha +\mu_2\right)x_2 \right) = \frac{-\partial_{s_1}\mu_1 +\omega_1\partial_{s_2} \mu_1}{x_2} - \frac{\mu'_2}{x_1} <0
\end{equation}
for all $\left(x_1,x_2\right)$ in the domain of $\varphi_0$. Thus, by the Dulac's Criterion \cite{kuzn}, there are no periodic orbits in the $x_1x_2$ face.

Now, consider system \eqref{eq:system_reduced} on the part of $\Omega$ with $x_1=0$, i.e.,

\begin{align}\label{x0_x2_system}
    \begin{cases}
        \begin{split}
            x_0' &= \left(-\alpha +\mu_0\left(-x_0 +u_f,-\omega_2x_0-x_2 +u_h\right)\right)x_0,\\
            x_2' &= \left(-\alpha + \mu_2\left(-\omega_2x_0 -x_2 +u_h\right)\right)x_2,
        \end{split}
    \end{cases}
\end{align}

defined on $\Omega_{02}$ given by

\begin{equation}\label{x0_x2_domain}
    \Omega_{02}=\left\{ \left(x_0,x_2\right)\in \mathbb{R}^2:  0\leq x_0 \leq \min\left\{u_f,\frac{u_h-x_2}{\omega_2}\right\}, 0\leq x_2\leq -\omega_2x_0+u_h \right\}.
\end{equation}

For auxiliary function
\begin{equation}
    \varphi_1 \left(x_0,x_2\right) = \frac{1}{x_0x_2}, \quad \left(x_0,x_2\right)\in \Omega_{02}\setminus \left( \left\{ x_0=0 \right\}\cup\left\{x_2=0\right\} \right)
\end{equation}

we have

\begin{equation}
    \nabla \cdot \left(\varphi_1\left(-\alpha +\mu_1\right)x_0 , \varphi_1\left(-\alpha +\mu_2\right)x_2 \right) = \frac{-\partial_{s_0}\mu_0 -\omega_2\partial_{s_2} \mu_0}{x_2} - \frac{\mu'_2}{x_0} <0.
\end{equation}

This, together with the fact that the axes $x_0=0$, $x_2=0$ are invariant, shows that there are no periodic orbits in the $x_0x_2$ face.

Finally, consider system \eqref{eq:system_reduced} on the part of $\Omega$ with $x_2=0$, i.e.,

\begin{align}\label{x0_x1_system}
    \begin{cases}
        \begin{split}
            x_0' &= \left(-\alpha +\mu_0\left(-x_0 +u_f,-\omega_2x_0+\omega_1x_1 +u_h\right)\right)x_0,\\
            x_1' &= \left(-\alpha +\mu_1\left(\omega_0x_0-x_1 +u_g,-\omega_2x_0+\omega_1x_1 +u_h\right)\right)x_1.
        \end{split}
    \end{cases}
\end{align}

defined on $\Omega_{01}$ given by

\begin{equation}
    \Omega_{01} =\left\{ \left(x_0,x_1\right)\in\mathbb{R}^2: 0\leq x_0 \leq u_f, \max\left\{0,\frac{\omega_2x_0-u_h}{\omega_1}\right\} \leq x_1 \leq \omega_0x_0 + u_g \right\}.
\end{equation}

Analogously to the previous cases, we define

\begin{equation}
    \varphi_2 \left(x_0,x_1\right) = \frac{1}{x_0x_1}, \quad \left(x_0,x_1\right)\in \Omega_{01}\setminus \left( \left\{ x_0=0 \right\}\cup\left\{x_1=0\right\} \right)
\end{equation}

and compute

\begin{equation}
    \nabla \cdot \left(\varphi\left(-\alpha +\mu_0\right)x_0 , \varphi\left(-\alpha +\mu_1\right)x_1 \right) = \frac{-\partial_{s_0}\mu_0 -\omega_2\partial_{s_2} \mu_0}{x_1} + \frac{-\partial_{s_1}\mu_1+\omega_1\partial_{s_2}\mu_1}{x_0} <0.
\end{equation}

Since the axes $x_0=0$, $x_1=0$ are invariant, by Dulac's Criterion, there are no periodic orbits in the $x_0x_1$ face.

\subsection{Hopf Bifurcation}\label{Hopf_bifurcation}
In this work we are especially interested in developing a more systematic approach to studying the Hopf bifurcation of the interior equilibrium. That Hopf bifurcation that occurs in this model was previously observed numerically \cite{sari2017}. The occurrence of a stable periodic orbit in system \eqref{eq:system_reduced} represents a situation in which all three populations of microorganisms oscillate indefinitely, and as a consequence, the substrate concentrations fluctuate. The characteristic polynomial of the Jacobian $\prescript{}{\left(111\right)}{J}$ corresponding to the interior equilibrium is given by

\begin{equation}
    \lambda^3 + a_2\lambda^2 + a_1\lambda + a_0 = 0,
\end{equation}
where
\begin{align}
    a_2 &= -\overset{*}{x}_0 \left( -\frac{\partial \mu_0}{\partial s_0} -\omega_2 \frac{\partial \mu_0}{\partial s_2}\right) - \overset{*}{x}_1\left(-\frac{\partial \mu_1}{\partial s_1} + \omega_1 \frac{\partial\mu_1}{\partial s_2}\right) +\overset{*}{x}_2 \mu_2',\\
    a_1 &= \overset{*}{x}_1 \frac{\partial \mu_1}{\partial s_1}\left(\overset{*}{x}_0\frac{\partial\mu_0}{\partial s_0}-\left(\omega_0\omega_1-\omega_2\right)\overset{*}{x}_0\frac{\partial \mu_0}{\partial s_2}+\overset{*}{x}_2\mu_2' \right)+\overset{*}{x}_0\frac{\partial \mu_0}{\partial s_0}\left(-\omega_1\overset{*}{x}_1\frac{\partial\mu_1}{\partial s_2}+\overset{*}{x}_2\mu_2'\right) ,\\
    a_0 &= \overset{*}{x}_0\overset{*}{x}_1\overset{*}{x}_2\frac{\partial \mu_0}{\partial s_0}\frac{\partial \mu_1}{\partial s_1}\mu_2',
\end{align}
and the coefficients $a_2$, $a_1$, and $a_0$ depend on the parameters $u_f$, $u_g$, $u_h$, and $\alpha$. The coefficients $a_2$ and $a_0$ are sign-definite (they are both positive), and $a_1$ might possibly change sign. Let us first notice that since the polynomial has order three, a real eigenvalue always exists. By the Routh-Hurwitz criterion, the above polynomial has a pair of purely imaginary eigenvalues if and only if \begin{equation}
    a_2a_1=a_0 \quad \text{(which implies $a_1>0$).}
\end{equation}
In that case we also have
\begin{equation}
    \lambda^3 + a_2\lambda^2 + a_1\lambda + a_0 = \lambda^3 + a_2\lambda^2 + a_1\lambda + a_2a_1 = \left(\lambda+a_2\right)\left(\lambda^2+a_1\right).
\end{equation}
Hence the eigenvalues are $\lambda_1=-a_2$ and $\lambda_{2,3}=\pm\sqrt{a_1}i$. Since eigenvalues are continuous functions of the parameters, we can see that if there is some $\left(u_f,u_g,u_h,\alpha\right)=\left(u_f^\ast,u_g^\ast,u_h^\ast,\alpha^\ast\right)$ such that $a_2a_1=a_0$, then if we denote by $\lambda_1$ the always present real eigenvalue, there is some $\delta>0$ such that for $\left\|\left(u_f,u_g,u_h,\alpha\right)-\left(u_f^\ast,u_g^\ast,u_h^\ast,\alpha^\ast\right)\right\|<\delta$ we always have $\lambda_1<0$. By lemma $5.1$, section $5.2$ in \cite{kuzn}, this implies the existence of a parameter-dependent, smooth, attracting, two dimensional, center manifold $W^c_{\left(u_f,u_g,u_h,\alpha\right)}$. In the following analysis, we consider only parameters that are in the $\delta$-neighborhood of $\left(u_f,u_g,u_h,\alpha\right)$ in order to ensure that the real eigenvalue $\lambda_1$ is negative.

By the Routh-Hurwitz criterion (since we just shown that when $\prescript{}{\left(111\right)}{\mathcal{E}}$ exists that
$a_0$ and $a_2$ are always positive), a Hopf bifurcation occurs when the expression $a_2a_1-a_0$ changes sign as a parameter varies. This ensures that  the real part of a pair of  complex eigenvalues with nonzero imaginary part passes through $0$ and hence changes sign.  This is related to the transversality condition:  the derivative of the real part of the eigenvalue with respect to the bifurcation parameter evaluated at the critical value when the real parts are zero is non-zero. We check this condition for specific forms of the functions $\mu_0$, $\mu_1$, and $\mu_2$. For the prototypes proposed in (\ref{eq:growth_fcns}), and with the values of parameters from Table \ref{table0} fixed, the function
\begin{equation}\label{eq:hopf_bif_fcn}
f\left(u_f,u_g,u_h,\alpha\right) = a_2a_1-a_0
\end{equation}
is an algebraic function in $u_f$, $u_g$, $u_h$, and $\alpha$. Hence fixing all the parameters except one makes the function $f$ a polynomial, the order of which depends on the choice of the free parameter. Specifically if we choose:
\begin{itemize}
    \item $\alpha$ - free parameter $\implies$ $f$ has order $27$,
    \item $u_f$ - free parameter $\implies$ $f$ has order $3$,
    \item $u_g$ - free parameter $\implies$ $f$ has order $3$,
    \item $u_h$ - free parameter $\implies$ $f$ has order $2$.
\end{itemize}
We construct bifurcation diagrams to explore the possibility of Hopf bifurcations as $\alpha$ varies in subsection \ref{subsec:bif_diag}, and begin our theoretical analysis by choosing $u_f$ as the free parameter. We have
\begin{equation}\label{eq:f_uf_fcn}
    f_{u_f}\left(u_f\right) = a_2a_1-a_0 = b_3u_f^3 + b_2u_f^2 + b_1u_f + b_0,
\end{equation}
and we assume that there is a value $u_f=u_f^\ast$ such that $f_{u_f}(u_f^\ast)=0$. We want to find  conditions on the coefficients of $f_{u_f}$ that guarantee that the derivative of $a_2 a_1-a_0$ with respect to $u_f$ is not equal to zero when $a_2 a_1-a_0=0$, i.e., that $u_f^\ast$ is not a local extremum of $f_{u_f}$. We have
\begin{equation}
    f'_{u_f}\left(u_f\right) = 3b_3u_f^2 + 2b_2u_f + b_1.
\end{equation}
The necessary condition for $u_f^\ast$ to be a local extremum for $f_{u_f}$ is $f'_{u_f}(u_f^\ast)=0$, that is
\begin{equation}\label{eq:u_f_loc_extr}
    u_f^\ast = \frac{-b_2 \pm \sqrt{b_2^2-3b_1b_3}}{3b_3}. 
\end{equation}
We can derive sufficient conditions for $u_f^\ast$ to be an extremum (for example by computing the second derivative of $f_{u_f}$), but the condition (\ref{eq:u_f_loc_extr}) is already very restrictive and will be sufficient for our work. We have thus obtained a sufficient condition for a Hopf bifurcation.

If we choose $u_g$ as the free parameter, we have
\begin{equation}\label{eq:f_ug_fcn}
    f_{u_g}\left(u_h\right) = c_3u_g^3 + c_2u_g^2 + c_1u_g + c_0,
\end{equation}
and with the assumption that $f_{u_g}\left(u_g^\ast\right)=0$, by a similar analysis as in the previous case, we obtain an analogous sufficient condition for a Hopf bifurcation in $u_g$.

Finally, if we choose $u_h$ as the free parameter, we have
\begin{equation}\label{eq:f_uh_fcn}
    f_{u_h}\left(u_h\right) = d_2u_h^2 + d_1u_h + d_0.
\end{equation}
Once again, we assume that $f_{u_h}\left(u_h^\ast\right)=0$, that is
\begin{equation}
    u_h^\ast = \frac{-d_1 \pm \sqrt{d_1^2-4d_2d_0}}{2d_2}.
\end{equation}
Here, $u_h^\ast$ is the local extremum if and only if the discriminant of equation (\ref{eq:f_uh_fcn}) is zero, i.e., if $d_1^2-4d_2d_0=0$.

We summarize our results in the following theorem.
\begin{theorem}\label{thm:Hopf}
Consider system (\ref{eq:system_reduced}) with the prototypes given by (\ref{eq:growth_fcns}) and with the values of parameters from Table \ref{table0} fixed. Assume that there exists a point $\left(u_f,u_g,u_h,\alpha\right)=\left(u_f^\ast,u_g^\ast,u_h^\ast,\alpha^\ast\right)$ such that $f\left(u_f^\ast,u_g^\ast,u_h^\ast,\alpha^\ast\right)=0$ for $f$ defined in (\ref{eq:hopf_bif_fcn}). Then $f_{u_f}\left(u_f\right)=f\left(u_f,u_g^\ast,u_h^\ast,\alpha^\ast\right)$, $f_{u_g}\left(u_g\right)=f\left(u_f^\ast,u_g,u_h^\ast,\alpha^\ast\right)$, and $f\left(u_h\right)=f\left(u_f^\ast,u_g^\ast,u_h,\alpha^\ast\right)$ are given by the equations (\ref{eq:f_uf_fcn}), (\ref{eq:f_ug_fcn}), and (\ref{eq:f_uh_fcn}), respectively. Also, there exists $\delta >0$ such that if  $\left\|\left(u_f,u_g,u_h,\alpha\right)-\left(u_f^\ast,u_g^\ast,u_h^\ast,\alpha^\ast\right)\right\|<\delta$, then
\begin{enumerate}[I.]
    \item if
    \begin{equation}
        u_f^\ast \neq \frac{-b_2 \pm \sqrt{b_2^2-3b_1b_3}}{3b_3},
    \end{equation}
    then there is a Hopf bifurcation in $u_f$ at $u_f=u_f^\ast$,
    \item if
    \begin{equation}
        u_g^\ast \neq \frac{-c_2 \pm \sqrt{c_2^2-3c_1c_3}}{3c_3},
    \end{equation}
    then there is a Hopf bifurcation in $u_g$ at $u_g=u_g^\ast$,
    \item if
    \begin{equation}
    d_1^2-4d_2d_0\neq0,
    \end{equation}
    then there is a Hopf bifurcation in $u_h$ at $u_h=u_h^\ast$.
\end{enumerate}
\end{theorem}

We now illustrate the theoretical results with the numerical simulations. To approximate values of the equilibria we used Maple software \cite{Maple:2018}, rounding all the values to $6$ significant digits. For the choice of parameters given in Table \ref{table0} and
\begin{equation}
    \alpha=0.01, \quad u_f=0.5, \quad u_g=0.0006,
\end{equation}
it follows that the only zero of \eqref{eq:f_uh_fcn} occurs for $u_h=0.102520$. In Figure~\ref{fig:Hopf_uh_before} we plot the phase space for $u_h=0.05$ (just before the Hopf bifurcation) using the ode15s solver from \cite{MATLAB:2018}. For this set of parameters, we have the following approximate values of the equilibria

\begin{equation}
\begin{split}
\prescript{}{\left(000\right)}{\mathcal{E}}&=\left(0,0,0\right) \quad \text{(unstable)},\\
\prescript{}{\left(100\right)}{\mathcal{E}}&=\left(0.000299015,0,0\right) \quad \text{(stable)},\\ \prescript{}{\left(001\right)}{\mathcal{E}}&=\left(0,0,0.0473693\right) \quad \text{(unstable)},\\ 
\prescript{}{\left(110\right)}{\mathcal{E}}^{\left(1\right)}&=\left(0.0545964, 0.00534486,0\right) \quad \text{(unstable)},\\ \prescript{}{\left(110\right)}{\mathcal{E}}^{\left(2\right)}&=\left(0.489465, 0.0487183,0\right) \quad \text{(unstable)},\\
\prescript{}{\left(111\right)}{\mathcal{E}}&=\left(0.306611,0.0520205,36.2277\right) \quad \text{(unstable)},
\end{split}
\end{equation}

\begin{figure}[H]
    \centering
    {\includegraphics[width=0.8\textwidth,height=0.6\textwidth]{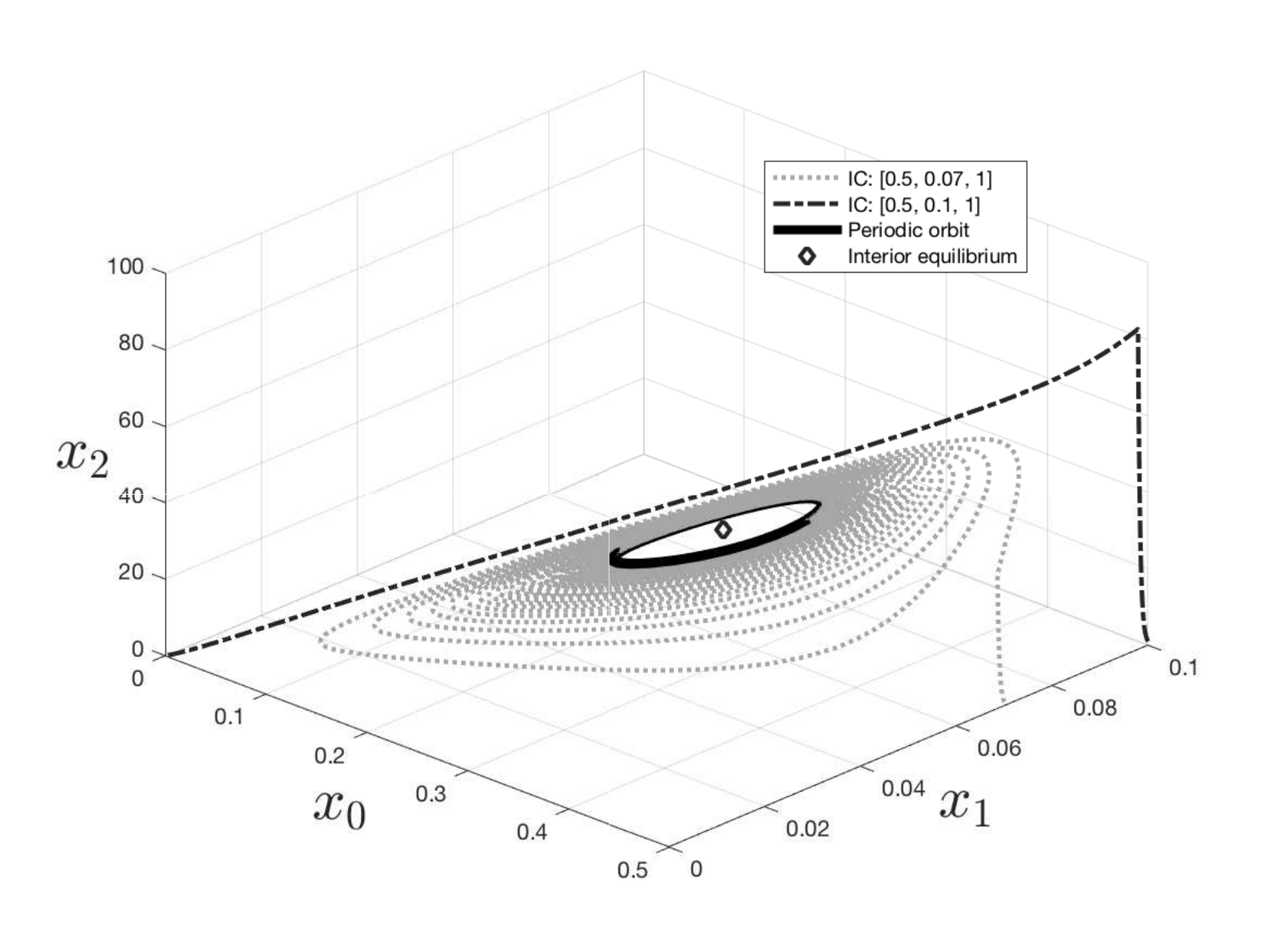}}
    \caption{Phase space of system \eqref{eq:system_reduced} with $\alpha=0.01$, $u_f=0.5$, $u_g=0.0006$, and $u_h=0.05$.}
    \label{fig:Hopf_uh_before}
\end{figure}

We can see that there exists a stable periodic orbit in the system, but depending on the initial conditions, the solution might also converge to the boundary equilibrium $\prescript{}{\left(100\right)}{\mathcal{E}}$. Thus in this case we observe bistability.

We now repeat the simulations for $u_h=0.3$ (after the predicted Hopf bifurcation), presented in Figure \ref{fig:Hopf_uh_after}. With all the other parameters set to the same values as in the previous case, we have the following equilibria

\begin{equation}
\begin{split}
\prescript{}{\left(000\right)}{\mathcal{E}}&=\left(0,0,0\right) \quad \text{(unstable)},\\
\prescript{}{\left(100\right)}{\mathcal{E}}&=\left(0.00183202,0,0\right) \quad \text{(stable)},\\ \prescript{}{\left(001\right)}{\mathcal{E}}&=\left(0,0,0.297369\right) \quad \text{(unstable)},\\ 
\prescript{}{\left(110\right)}{\mathcal{E}}^{\left(1\right)}&=\left(0.0528596,0.00502299,0\right) \quad \text{(unstable)},\\ \prescript{}{\left(110\right)}{\mathcal{E}}^{\left(2\right)}&=\left(0.489467,0.0485697,0\right) \quad \text{(unstable)},\\
\prescript{}{\left(111\right)}{\mathcal{E}}&=\left(0.306611,0.0520205,36.4777\right) \quad \text{(stable)}.
\end{split}
\end{equation}

\begin{figure}[H]
    \centering
    {\includegraphics[width=0.8\textwidth,height=0.6\textwidth]{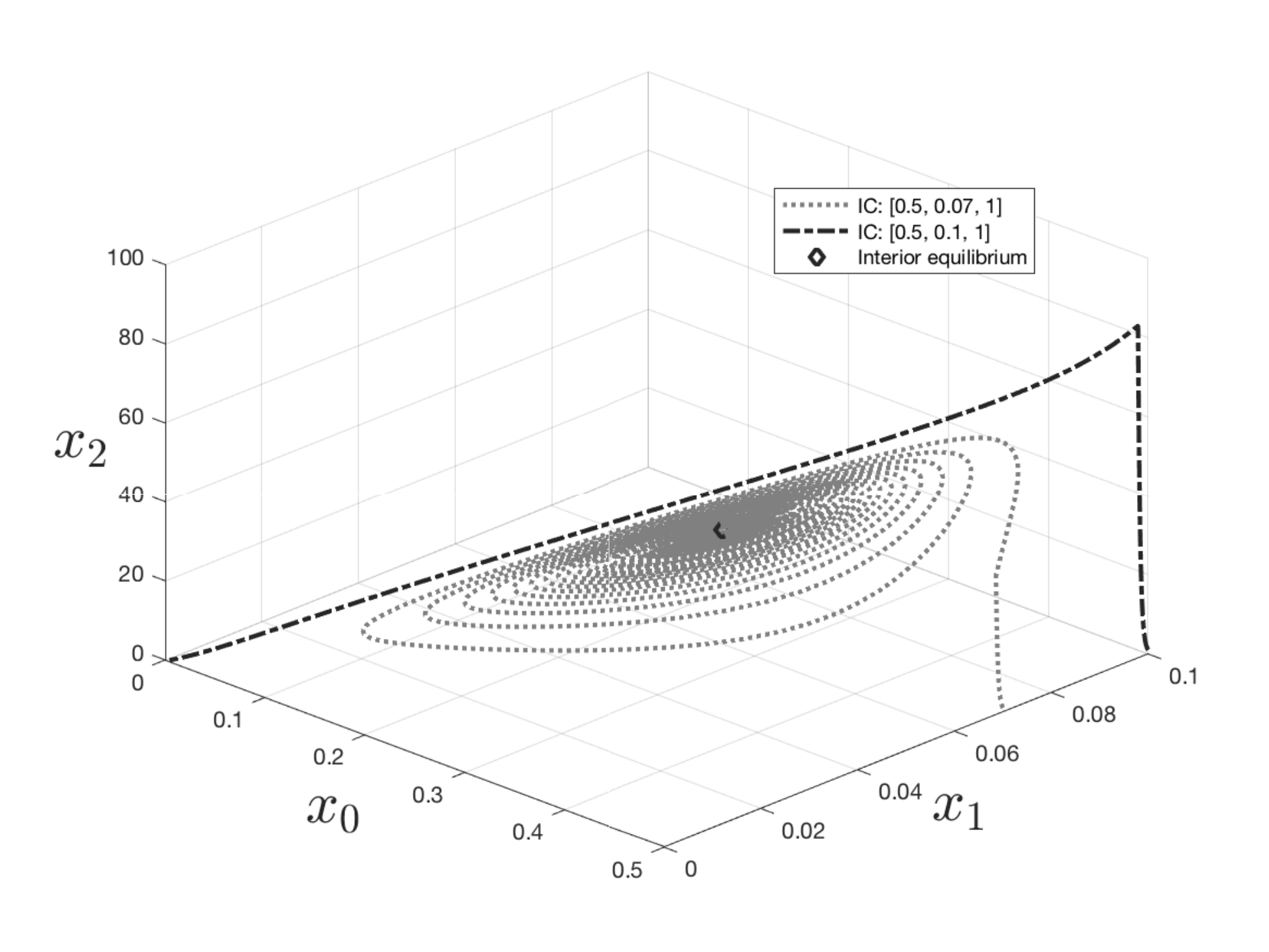}}
    \caption{Phase space of system \eqref{eq:system_reduced} with $\alpha=0.01$, $u_f=0.5$, $u_g=0.0006$, and $u_h=0.3$.}
    \label{fig:Hopf_uh_after}
\end{figure}

We can see that the Hopf bifurcation occurs between $u_h=0.05$ and $u_h=0.3$. The stable periodic orbit is no longer present, and the interior equilibrium is stable. Once again, the boundary equilibrium $\prescript{}{\left(100\right)}{\mathcal{E}}$ is stable, and thus we observe bistability in the system.

It is worth to notice, that increasing $u_h$ had a stabilizing effect on the system (actually, this type of behaviour applies also to $u_f$ and $u_g$). This result is especially important in the context of the modeled phenomenon, since the most desirable situation happens when the  production of methane is not fluctuating. Variable rates of gas production can result in decreased productivity of the biogas plant.

\subsection{Persistence}
The notion of persistence is particularly important in modeling biological phenomena. Roughly speaking, we say that a system is persistent if all the species with positive initial populations survive. The formal definition is as follows.

\begin{definition}
The system
\begin{equation}
\begin{split}
    x_i'&=x_if_i\left(x_1,x_2,\ldots,x_n\right),\\
    x_i\left(0\right) &= x_{i0}\geq 0, \quad i=1,2,\ldots,n,
\end{split}
\end{equation}
is said to be weakly persistent if
\begin{equation}
    \limsup\limits_{t\rightarrow\infty}x_i\left(t\right)>0, \quad i=1,2,\ldots,n
\end{equation}
for every trajectory with positive initial conditions, and is said to be persistent if
\begin{equation}
    \liminf\limits_{t\rightarrow\infty}x_i\left(t\right)>0, \quad i=1,2,\ldots,n
\end{equation}
for every trajectory with positive initial conditions. This system is said to be uniformly persistent if there exists a positive number $\epsilon$ such that
\begin{equation}
    \liminf\limits_{t\rightarrow\infty}x_i\left(t\right)\geq \epsilon, \quad i=1,2,\ldots,n
\end{equation}
for every trajectory with positive initial conditions.
\end{definition}
To prove that system \eqref{eq:system_reduced} is persistent, we will use the Butler-McGehee lemma \cite{smith_waltman_1995} repeatedly.
\begin{lemma}\label{lem:butler_mcgehee}
Suppose that $x^\ast$ is a hyperbolic equilibrium point of the system
\begin{equation}
\begin{split}
    x' &= f(x),\\
    x\left(0\right)&=x_0,
\end{split}
\end{equation}
with $x\in\mathbb{R}^n$ and $f:\mathbb{R}^n\to\mathbb{R}^n$, where $f$ is continuously differentiable. Suppose also that $x^\ast$ is in $\omega\left(x_0\right)$, the omega limit set of $\gamma^+\left(x_0\right)$ (the positive semi-orbit through $x_0$), but is not the entire omega limit set. Then $\omega\left(x_0\right)$ has nontrivial (i.e., different from $x^\ast$) intersection with the stable and unstable manifolds of $x^\ast$.
\end{lemma}
As we already noticed in section \ref{Hopf_bifurcation}, there are values of the parameters where one of the boundary equilibria and the interior equilibrium are both asymptotically stable, and hence system \eqref{eq:system_reduced} is not persistent, even though  an interior equilibrium point exists. We will thus focus on the cases for which no boundary equilibrium point of system \eqref{eq:system_reduced} is stable.

\begin{theorem}\label{thm:persistence1}
Let system \eqref{eq:system_reduced} have the following equilibria configuration (as represented schematically in Figure \ref{fig:persist_schem1}):
\begin{center}
    \begin{tabular}{C{2cm} C{6cm} C{6cm}}
        \hline
        Equilibrium  & Number of eigenvalues with positive real part & Number of eigenvalues with negative real part \\
        \hline
        $\prescript{}{\left(000\right)}{\mathcal{E}}$ & 2 & 1 \\
        $\prescript{}{\left(100\right)}{\mathcal{E}}$ & 1 & 2 \\
        $\prescript{}{\left(001\right)}{\mathcal{E}}$ & 1 & 2 \\
        $\prescript{}{\left(110\right)}{\mathcal{E}}$ & 1 & 2 \\
        \hline
    \end{tabular}
\end{center}
Then system \eqref{eq:system_reduced} is persistent.

\end{theorem}

\begin{figure}[H]
    \centering
    {\includegraphics[width=0.9\textwidth,height=0.7\textwidth]{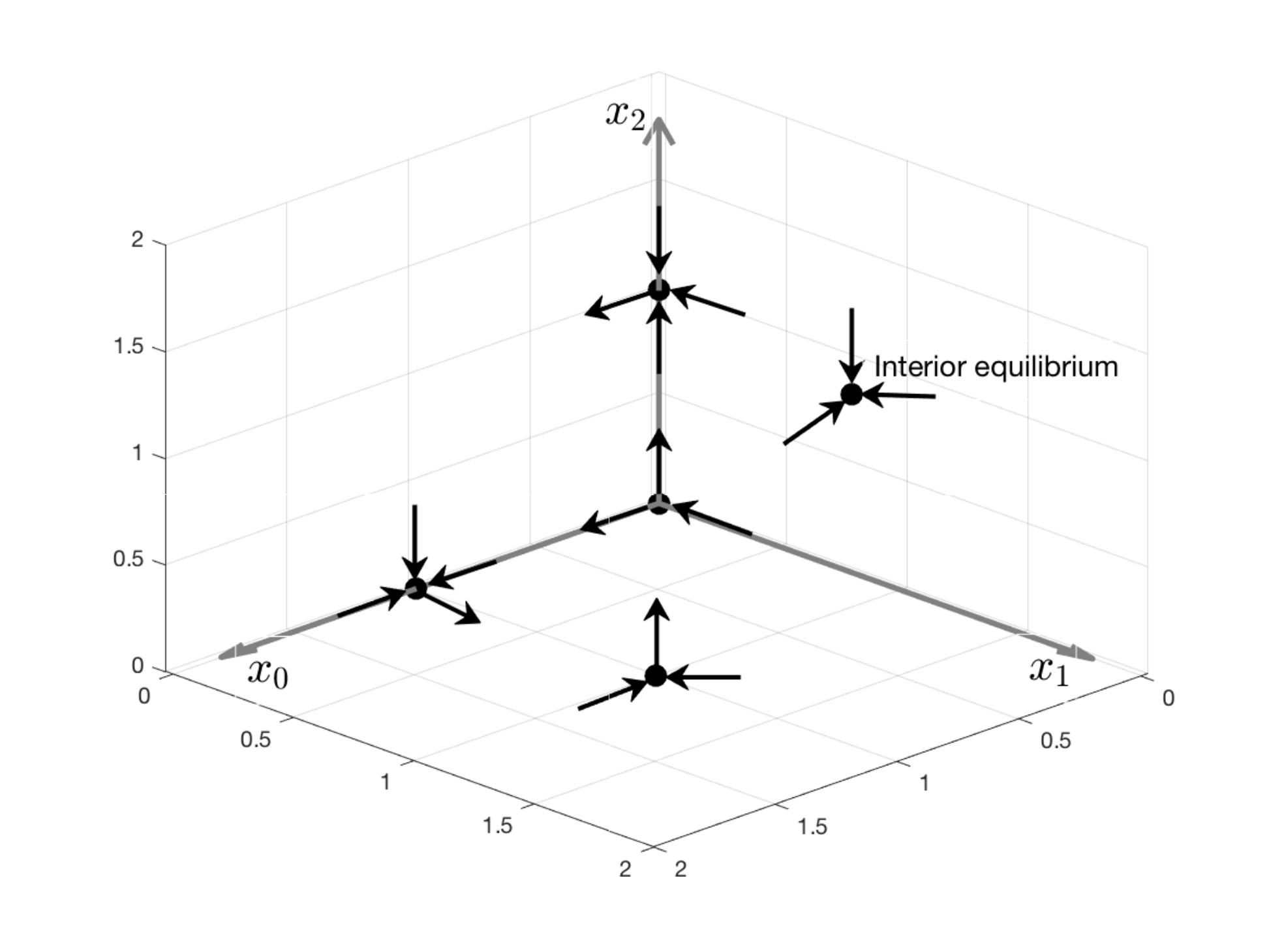}}
    \caption{Schematic representation of the equilibria configuration occurring in the hypothesis of theorem \ref{thm:persistence1}. Black arrows represent stable and unstable manifolds of each of the equilibrium (marked by the orange dots). In the example for the parameters we select in model \eqref{eq:system_reduced} to illustrate this theorem (see \eqref{eq:perst_par1}), there is an asymptotically stable interior equilibrium (as shown). However, this is not necessary for the proof of theorem \ref{thm:persistence1}.}
    \label{fig:persist_schem1}
\end{figure}

\begin{proof}
Since planes $x_0x_1$, $x_0x_2$ and $x_1x_2$ are invariant, we know where the stable and unstable manifolds of the boundary equilibria lie. This is represented in a schematic way in Figure \ref{fig:persist_schem1}. Keeping this picture in mind should make the following argument much more transparent. Assume that a solution $\Vec{x}\left(t\right) = \left(x_0\left(t\right),x_1\left(t\right),x_2\left(t\right)\right)$ with an initial condition $\Vec{x}^{\left(0\right)}= \left(x_0^{\left(0\right)},x_1^{\left(0\right)},x_2^{\left(0\right)}\right)$, where $x_i^{\left(0\right)}>0$, $i=1,2,3$, is given. First, suppose that $\prescript{}{\left(000\right)}{\mathcal{E}}$ belongs to $\omega\left(\gamma^+\left(\Vec{x}^{\left(0\right)}\right)\right)$, the omega limit set of $\gamma^+ \left(\Vec{x}^{\left(0\right)}\right)$. Since $\prescript{}{\left(000\right)}{\mathcal{E}}$ is a saddle point with one-dimensional stable manifold restricted to the $x_1$-axis, it is not the entire omega limit set $\omega\left(\gamma^+\left(\Vec{x}^{\left(0\right)}\right)\right)$. Hence, by lemma \ref{lem:butler_mcgehee}, there is a point $x^\ast\neq\prescript{}{\left(000\right)}{\mathcal{E}}$ in both $\omega\left(\gamma^+\left(\Vec{x}^{\left(0\right)}\right)\right)$ and $W^s\left(\prescript{}{\left(000\right)}{\mathcal{E}}\right)$, the stable manifold of $\prescript{}{\left(000\right)}{\mathcal{E}}$. The entire orbit through any point in an omega limit set is also in the omega limit set. The stable manifold of $\prescript{}{\left(000\right)}{\mathcal{E}}$ is the $x_1$-axis, and the $x_1$-axis is unbounded. We have already proven in section \ref{sec:reduction} that all orbits of system \eqref{eq:system_reduced} are bounded, and hence the omega limit set of any orbit of \eqref{eq:system_reduced} is bounded. This contradicts the existence of such an $x^\ast$ and thus $\prescript{}{\left(000\right)}{\mathcal{E}}\notin\omega\left(\gamma^+\left(\Vec{x}^{\left(0\right)}\right)\right)$.

Now, suppose that $\prescript{}{\left(001\right)}{\mathcal{E}}\in\omega\left(\gamma^+\left(\Vec{x}^{\left(0\right)}\right)\right)$. Since $\prescript{}{\left(001\right)}{\mathcal{E}}$ is a saddle point with two-dimensional stable manifold restricted to the $x_0x_1$-plane, $\left\{\prescript{}{\left(001\right)}{\mathcal{E}}\right\}$ is not the entire omega limit set $\omega\left(\gamma^+\left(\Vec{x}^{\left(0\right)}\right)\right)$. Thus, using lemma \ref{lem:butler_mcgehee}, there is a point $x^\ast\in \omega\left(\gamma^+\left(\Vec{x}^{\left(0\right)}\right)\right) \cap W^s\left(\prescript{}{\left(001\right)}{\mathcal{E}}\right)\setminus\left\{\prescript{}{\left(001\right)}{\mathcal{E}}\right\}$. Since the stable manifold $W^s\left(\prescript{}{\left(001\right)}{\mathcal{E}}\right)$ lies entirely in the $x_1x_2$-plane, and the entire orbit through $x^\ast$ is in $\omega\left(\gamma^+\left(\Vec{x}^{\left(0\right)}\right)\right)$, by the analysis in subsection \ref{subsec:faces}, this orbit becomes unbounded in backward time. This contradiction shows that $\prescript{}{\left(001\right)}{\mathcal{E}}\notin\omega\left(\gamma^+\left(\Vec{x}^{\left(0\right)}\right)\right)$.

Now, suppose that $\prescript{}{\left(100\right)}{\mathcal{E}}\in\omega\left(\gamma^+\left(\Vec{x}^{\left(0\right)}\right)\right)$. Similarly as in the previous cases, this implies that there exists a point $x^\ast\in \omega\left(\gamma^+\left(\Vec{x}^{\left(0\right)}\right)\right) \cap W^s\left(\prescript{}{\left(100\right)}{\mathcal{E}}\right)\setminus\left\{\prescript{}{\left(100\right)}{\mathcal{E}}\right\}$. This time the stable manifold $W^s\left(\prescript{}{\left(100\right)}{\mathcal{E}}\right)$ is two-dimensional and lies entirely in the $x_0x_2$-plane. By the analysis in subsection \ref{subsec:faces}, the entire orbit through $x^\ast$ \Big(which belongs to $\omega\left(\gamma^+\left(\Vec{x}^{\left(0\right)}\right)\right)$\Big) becomes unbounded in backward time or its closure contains $\prescript{}{\left(001\right)}{\mathcal{E}}$. This contradiction proves that  $\prescript{}{\left(100\right)}{\mathcal{E}}\notin\omega\left(\gamma^+\left(\Vec{x}^{\left(0\right)}\right)\right)$.

Now, suppose that $\prescript{}{\left(110\right)}{\mathcal{E}}\in\omega\left(\gamma^+\left(\Vec{x}^{\left(0\right)}\right)\right)$. Again, $\left\{\prescript{}{\left(110\right)}{\mathcal{E}}\right\}$ is not the entire omega limit set $\omega\left(\gamma^+\left(\Vec{x}^{\left(0\right)}\right)\right)$, so there exists a point $x^\ast\in\omega\left(\gamma^+\left(\Vec{x}^{\left(0\right)}\right)\right)\cap W^s\left(\prescript{}{\left(110\right)}{\mathcal{E}}\right)\setminus\left\{\prescript{}{\left(110\right)}{\mathcal{E}}\right\}$. This point lies in the $x_0x_1$-plane, since $W^s\left(\prescript{}{\left(110\right)}{\mathcal{E}}\right)$ is two-dimensional and is entirely contained in this plane. As in the previous cases, the entire orbit through $x^\ast$ is in $\omega\left(\gamma^+\left(\Vec{x}^{\left(0\right)}\right)\right)$. Since there are no periodic orbits in the $x_0x_1$ face, and since $\left\{\prescript{}{\left(100\right)}{\mathcal{E}}\right\}\notin \omega\left(\gamma^+\left(\Vec{x}^{\left(0\right)}\right)\right)$, the orbit becomes unbounded in backward time. This contradiction proves that $\prescript{}{\left(110\right)}{\mathcal{E}}\notin\omega\left(\gamma^+\left(\Vec{x}^{\left(0\right)}\right)\right)$. 

Finally, consider any $\widehat{x}=\left(\widehat{x}_0,\widehat{x}_1,\widehat{x}_2\right)$, such that $\widehat{x}_i=0$ for at least one $i=1,2,3$, and suppose that $\widehat{x}\in\omega\left(\gamma^+\left(\Vec{x}^{\left(0\right)}\right)\right)$. Then, the entire orbit through $\widehat{x}$ is in $\omega\left(\gamma^+\left(\Vec{x}^{\left(0\right)}\right)\right)$. But since this orbit lies entirely in either $x_0x_1$, $x_1x_2$, or $x_0x_2$ face, it converges to one of the boundary equilibria. This implies that this boundary equilibrium is in $\omega\left(\gamma^+\left(\Vec{x}^{\left(0\right)}\right)\right)$, and this possibility has been eliminated in the previous part of the proof.

We have therefore proven that
\[
    \liminf\limits_{t\rightarrow\infty}x_i\left(t\right)>0, \quad i=1,2,3,
\]
i.e., that system \eqref{eq:system_reduced} is persistent. 
\end{proof}

An example satisfying the assumptions of theorem \ref{thm:persistence1} occurs for

\begin{equation}\label{eq:perst_par1}
    \alpha = 0.0002, \quad u_f=0.6, \quad u_g=0, \quad u_h=0.1.
\end{equation}

Persistence can also be observed with the addition of phenol, i.e., with $u_g>0$.

\begin{theorem}\label{thm:persistence2}
Let system \eqref{eq:system_reduced} have the following equilibria configuration (as represented schematically in Figure \ref{fig:persist_schem2}):

\begin{center}
    \begin{tabular}{C{2cm} C{6cm} C{6cm}}
        \hline
        Equilibrium  & Number of eigenvalues with positive real part & Number of eigenvalues with negative real part \\
        \hline
        $\prescript{}{\left(000\right)}{\mathcal{E}}$ & 2 & 1 \\
        $\prescript{}{\left(100\right)}{\mathcal{E}}$ & 1 & 2 \\
        $\prescript{}{\left(001\right)}{\mathcal{E}}$ & 1 & 2 \\
        $\prescript{}{\left(011\right)}{\mathcal{E}}$ & 1 & 2
        \\
        $\prescript{}{\left(110\right)}{\mathcal{E}}$ & 1 & 2 \\
        \hline
    \end{tabular}
\end{center}

Then system \eqref{eq:system_reduced} is persistent.
\end{theorem}

\begin{figure}[H]
    \centering
    {\includegraphics[width=0.9\textwidth,height=0.7\textwidth]{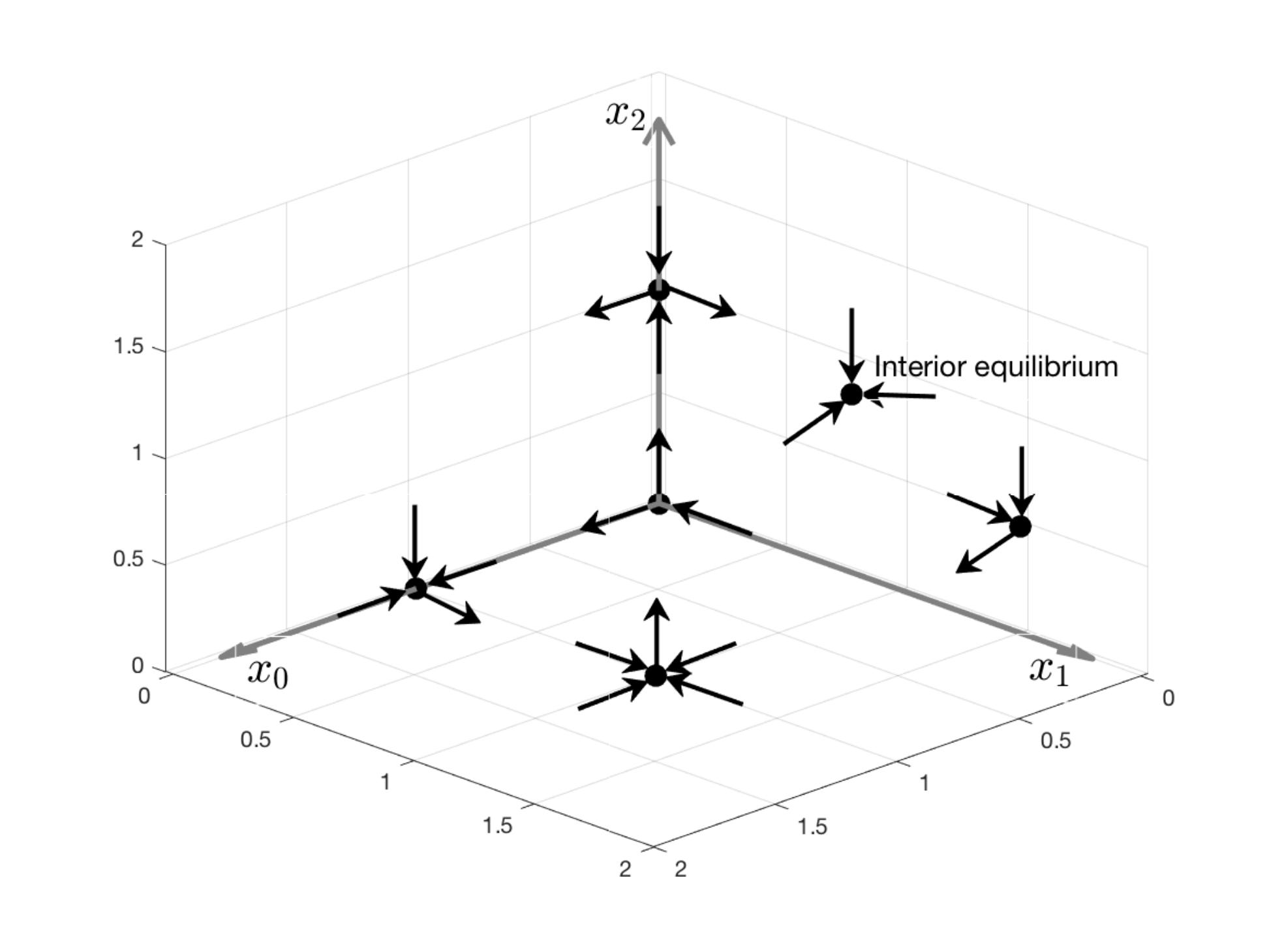}}
    \caption{Schematic representation of the equilibria configuration occurring in the hypothesis of theorem \ref{thm:persistence2}. Black arrows represent stable and unstable manifolds of each of the equilibrium (marked by the orange dots). In the example for the parameters we select in model \eqref{eq:system_reduced} to illustrate this theorem (see \eqref{eq:perst_par2}), there is an asymptotically stable interior equilibrium (as shown). However, this is not necessary for the proof of theorem \ref{thm:persistence2}.}
    \label{fig:persist_schem2}
\end{figure}

\begin{proof}
The idea behind this proof is very similar to the method presented in the proof of theorem \ref{thm:persistence1}. Let $\Vec{x}\left(t\right) = \left(x_0\left(t\right),x_1\left(t\right),x_2\left(t\right)\right)$ be a solution of \eqref{eq:system_reduced} with an initial condition $\Vec{x}^{\left(0\right)}= \left(x_0^{\left(0\right)},x_1^{\left(0\right)},x_2^{\left(0\right)}\right)$, where $x_i^{\left(0\right)}>0$, $i=1,2,3$. Since the stable and unstable manifolds of the all of the equilibria, except $\prescript{}{\left(001\right)}{\mathcal{E}}$ and $\prescript{}{\left(011\right)}{\mathcal{E}}$, have the same configuration as in the hypothesis of theorem \ref{thm:persistence1}, the argument eliminating them from the omega limit set of $\gamma^+\left(\Vec{x}^{\left(0\right)}\right)$ is exactly the same and we only need to focus on $\prescript{}{\left(001\right)}{\mathcal{E}}$ and $\prescript{}{\left(011\right)}{\mathcal{E}}$ equilibria.

Suppose that $\prescript{}{\left(001\right)}{\mathcal{E}}\in\omega\left(\gamma^+\left(\Vec{x}^{\left(0\right)}\right)\right)$. Since $\prescript{}{\left(001\right)}{\mathcal{E}}$ is a saddle point with one-dimensional stable manifold restricted to the $x_2$-axis, we have $\omega\left(\gamma^+\left(\Vec{x}^{\left(0\right)}\right)\right)\setminus\left\{\prescript{}{\left(001\right)}{\mathcal{E}}\right\}\neq\emptyset$. Hence, by lemma \ref{lem:butler_mcgehee}, there is a point $x^\ast\in\omega\left(\gamma^+\left(\Vec{x}^{\left(0\right)}\right)\right)\cap W^s\left(\prescript{}{\left(001\right)}{\mathcal{E}}\right)\setminus\left\{\prescript{}{\left(001\right)}{\mathcal{E}}\right\}$. The entire orbit through $x^\ast$, which also belongs to $\omega\left(\gamma^+\left(\Vec{x}^{\left(0\right)}\right)\right)$, either becomes unbounded in backward time, or converges to the $\prescript{}{\left(000\right)}{\mathcal{E}}$ equilibrium. Since all orbits of system \eqref{eq:system_reduced} are bounded, and $\prescript{}{\left(000\right)}{\mathcal{E}}\notin\omega\left(\gamma^+\left(\Vec{x}^{\left(0\right)}\right)\right)$, we obtain a contradiction. Hence $\prescript{}{\left(001\right)}{\mathcal{E}}\notin\omega\left(\gamma^+\left(\Vec{x}^{\left(0\right)}\right)\right)$.

Now, suppose that $\prescript{}{\left(011\right)}{\mathcal{E}}\in\omega\left(\gamma^+\left(\Vec{x}^{\left(0\right)}\right)\right)$. Since $\prescript{}{\left(011\right)}{\mathcal{E}}$ is a saddle point with two-dimensional stable manifold restricted to the $x_1x_2$-plane (it is repelling into the interior), we have $\omega\left(\gamma^+\left(\Vec{x}^{\left(0\right)}\right)\right)\setminus\left\{\prescript{}{\left(011\right)}{\mathcal{E}}\right\}\neq\emptyset$. By using lemma \ref{lem:butler_mcgehee}, there exists a point $x^\ast\in\omega\left(\gamma^+\left(\Vec{x}^{\left(0\right)}\right)\right)\cap W^s\left(\prescript{}{\left(011\right)}{\mathcal{E}}\right) \setminus\left\{\prescript{}{\left(011\right)}{\mathcal{E}}\right\}$. The entire orbit through $x^\ast$, which also belongs to $\omega\left(\gamma^+\left(\Vec{x}^{\left(0\right)}\right)\right)$, either becomes unbounded in backward time, or converges to $\prescript{}{\left(000\right)}{\mathcal{E}}$, or $\prescript{}{\left(001\right)}{\mathcal{E}}$ (we have previously shown in subsection \ref{subsec:faces} that there are no periodic orbits in the $x_1x_2$ face). Since we have already proven that $\prescript{}{\left(000\right)}{\mathcal{E}}\notin\omega\left(\gamma^+\left(\Vec{x}^{\left(0\right)}\right)\right)$, and $\prescript{}{\left(001\right)}{\mathcal{E}}\notin\omega\left(\gamma^+\left(\Vec{x}^{\left(0\right)}\right)\right)$, we obtain a contradiction, which proves that $\prescript{}{\left(011\right)}{\mathcal{E}}\notin\omega\left(\gamma^+\left(\Vec{x}^{\left(0\right)}\right)\right)$.

Finally, consider any $\widehat{x}=\left(\widehat{x}_0,\widehat{x}_1,\widehat{x}_2\right)$, such that $\widehat{x}_i=0$ for at least one $i=1,2,3$, and suppose that $\widehat{x}\in\omega\left(\gamma^+\left(\Vec{x}^{\left(0\right)}\right)\right)$. Then, the entire orbit through $\widehat{x}$ is in $\omega\left(\gamma^+\left(\Vec{x}^{\left(0\right)}\right)\right)$. But since this orbit lies entirely in either $x_0x_1$, $x_1x_2$, or $x_0x_2$ face, it converges to one of the boundary equilibria. This implies that this boundary equilibrium is in $\omega\left(\gamma^+\left(\Vec{x}^{\left(0\right)}\right)\right)$, and this possibility has been eliminated in the previous part of the proof.
\end{proof}

An example satisfying the assumptions of theorem \ref{thm:persistence2} occurs for
\begin{equation}\label{eq:perst_par2}
    \alpha = 0.0002, \quad u_f=0.6, \quad u_g=0.00015, \quad u_h=0.1.
\end{equation}

\begin{remark}
Interestingly enough, in many cases of models describing biological phenomena, persistence already implies uniform persistence. The rigorous results were obtained in \cite{butler_1986}. In our context, the key theorem from \cite{butler_1986} states that if $\mathcal{F}$ is a dynamical system for which $\mathbb{R}^n_+$ and $\partial \mathbb{R}^n_+$ are invariant, then $\mathcal{F}$ is uniformly persistent provided that

\begin{enumerate}
    \item $\mathcal{F}$ is dissipative (meaning that $\forall x\in \mathbb{R}^n_+$ $\omega\left(x\right)\neq\emptyset$ and $\bigcup_{x\in\mathbb{R}^n_+}\omega\left(x\right)$ has compact closure),
    \item $\mathcal{F}$ is weakly persistent,
    \item $\partial\mathcal{F}$ (the restriction of $\mathcal{F}$ to the boundary $\partial \mathbb{R}^n_+$) is "isolated",
    \item $\partial\mathcal{F}$ is "acyclic".
\end{enumerate}
These results can be easily modified so that we consider the flow $\mathcal{F}$ on $\Omega$ defined in \eqref{eq:inv_set}. Although $\partial\Omega$ is not invariant, the theorem from \cite{butler_1986}, as explained in \cite{BUTLER1987781}, can be modified so that it applies in the case when $\partial\Omega$ is the union of two sets $\Omega_1$ and $\Omega_2$, for which $\mathcal{F}$ is invariant on $\Omega_1$ and $\Omega_2$ is repelling into the interior of $\Omega$, provided that conditions $3.$ and $4.$ are satisfied for the restriction of $\mathcal{F}$ to $\Omega_1$. In our case, the positively invariant set $\Omega$, on which we analyze system \eqref{eq:system_reduced} is bounded, hence condition $1.$ is satisfied. Condition $2.$ holds by theorem \ref{thm:persistence1} (persistence implies weak persistence). In our context condition 3. is satisfied, because all the boundary equilibria are hyperbolic, and hence each one is the maximal invariant set in a neighbourhood of itself. Also, their union forms a covering of the omega limit sets of $\Omega_1$. Condition 4. is satisfied because the boundary equilibria are not cyclically linked, i.e., there is no cyclic chain of heteroclinic orbits joining them. Thus, we have shown not only persistence, but also uniform persistence of system \eqref{eq:system_reduced} in the case of theorem \ref{thm:persistence1} and theorem \ref{thm:persistence2}.
\end{remark}

We have thus proven the following theorem:

\begin{theorem}\label{thm:uniform_persistence}
Under the hypotheses of theorems \ref{thm:persistence1} and \ref{thm:persistence2}, system \eqref{eq:system_reduced} is uniformly persistent.
\end{theorem}

We finish this subsection by extending the uniform persistence to the original six dimensional system \eqref{eq:system_scaled}. Notice that if $\Vec{X_0}=\left(x_0(0),x_1(0),x_2(0),s_0(0),s_1(0),s_2(0)\right)$ with $x_i(0)\geq0$, $s_i(0)\geq0$, $i=1,2,3$, then we necessarily must have $\omega\left(\Vec{X_0}\right)\in\Omega$. Otherwise, there would exist a point $\Vec{p}\in\mathbb{R}_+\setminus \Omega$ and a sequence of times $\left(t_n\right)$ with $t_n\rightarrow\infty$ for which the corresponding solution converges to $\Vec{p}$. This would mean that $\Omega$ is not globally attracting, which was proven in section \ref{sec:reduction}. Also, if $\omega\left(\Vec{X_0}\right)$ has a point lying in a face with one of the $x_i$, $i=1,2,3$ equal zero, then the entire orbit through that point would also be in $\omega\left(\Vec{X_0}\right)$. Thus, if the assumptions of either theorem \ref{thm:persistence1} or theorem \ref{thm:persistence2} hold, the omega limit set $\omega\left(\Vec{X_0}\right)$ is entirely contained in the interior of $\Omega$. We have thus proven the following theorem:

\begin{theorem}\label{thm:uniform_persistence_6_d_system}
Under the hypotheses of theorems \ref{thm:persistence1} and \ref{thm:persistence2}, the six dimensional system \eqref{eq:system_scaled} is uniformly persistent.
\end{theorem}

\subsection{Bifurcation diagrams}\label{subsec:bif_diag}
As previously stated in section \ref{Hopf_bifurcation}, we now study numerically effects on the qualitative behaviour of system \eqref{eq:system_reduced} when considering $\alpha$ as the bifurcation parameter. Throughout this section, we assume that parameters $\omega_0$, $\omega_1$, $\omega_2$, $\phi_1$, $\phi_2$, $K_P$, and $K_I$ are fixed at the values given in Table \ref{table0}. 

We now fix the following parameters 
\begin{equation}
    u_f=2 , \quad u_g=0 , \quad u_h=0,
\end{equation}
and plot a one-parameter bifurcation diagram in $\alpha$, with $x_0$ on the $y$-axis. All simulations were performed using \cite{xpp}.

\begin{figure}[H]
    \centering
    {\includegraphics[width=0.9\textwidth,height=0.7\textwidth]{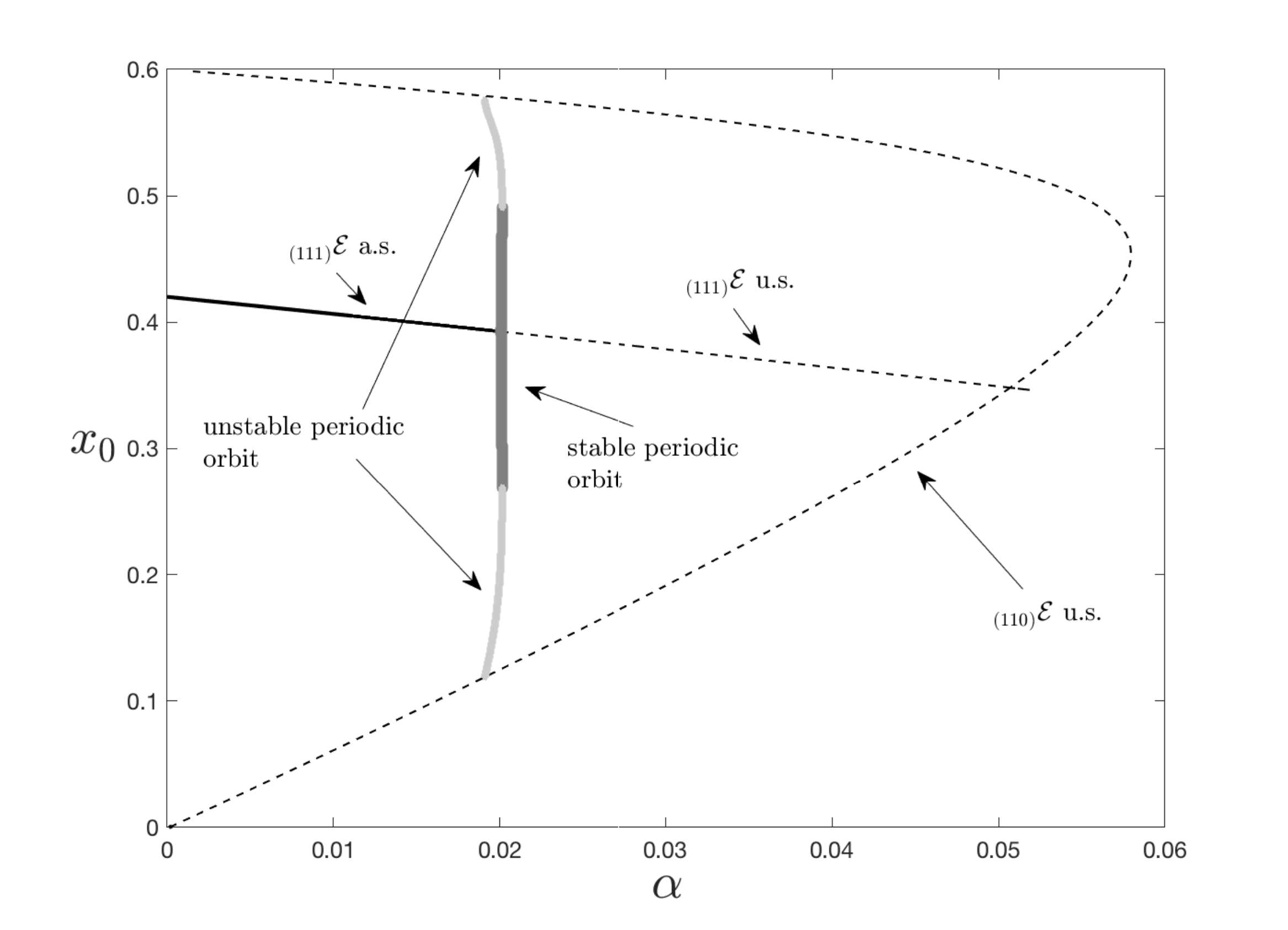}}
    \caption{One-parameter bifurcation diagram of system \eqref{eq:system_reduced} with $\alpha$ as the bifurcation parameter and $u_f=2$, $u_g=0$, $u_h=0$.}
    \label{fig:bif_diag_alph_x0}
\end{figure}

\begin{figure}[H]
    \centering
    {\includegraphics[width=0.9\textwidth,height=0.7\textwidth]{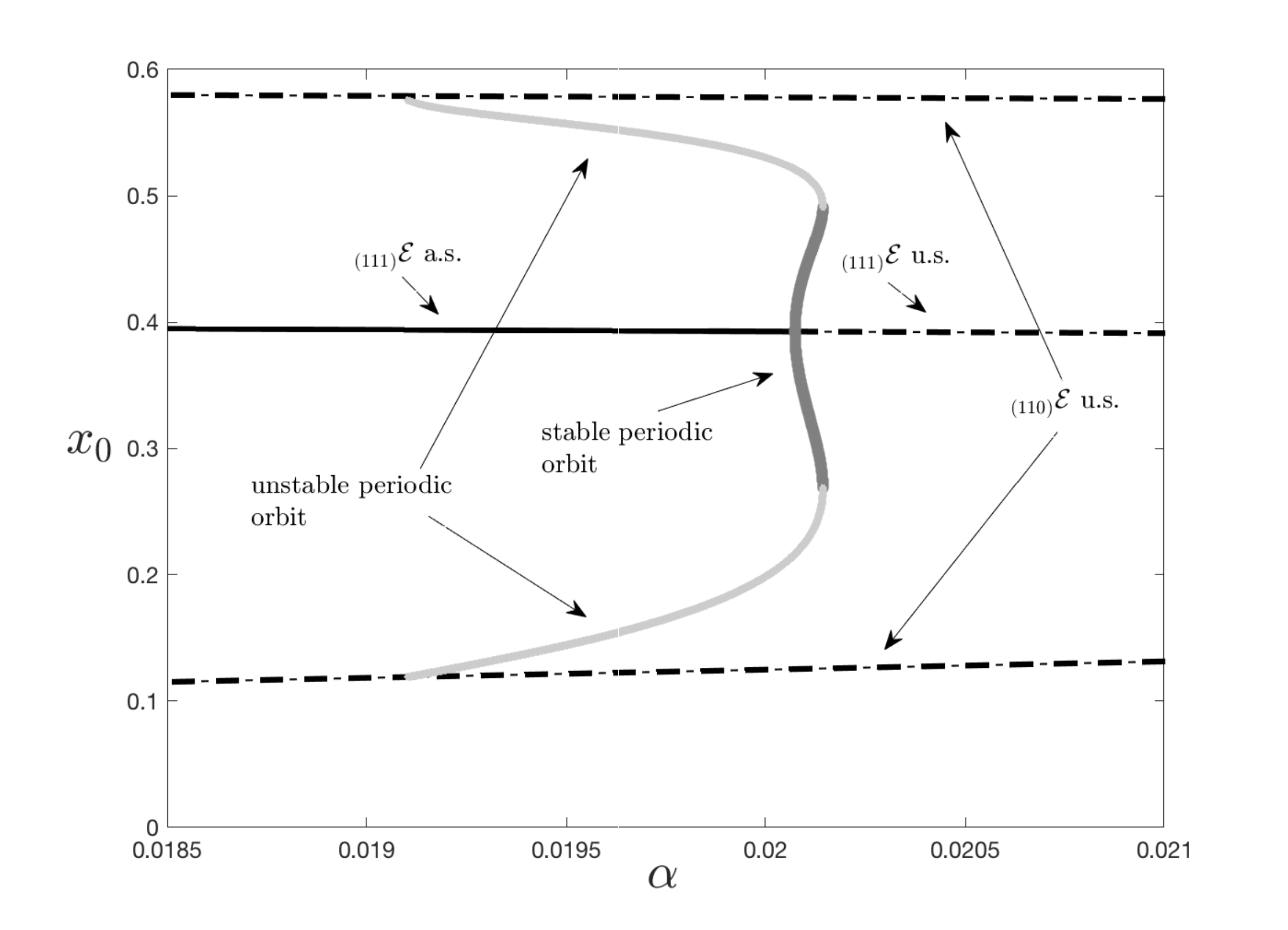}}
    \caption{Close-up of on the one-parameter bifurcation diagram represented in Figure \ref{fig:bif_diag_alph_x0}.}
    \label{fig:bif_diag_alph_x0_zoom}
\end{figure}

We can see that as $\alpha$ decreases, there is a saddle-node bifurcation, resulting in two equilibria $\prescript{}{\left(110\right)}{\mathcal{E}}^{\left(1\right)}$ and $\prescript{}{\left(110\right)}{\mathcal{E}}^{\left(2\right)}$ appearing (both unstable). Next, there is a transcritical bifurcation with the $\prescript{}{\left(110\right)}{\mathcal{E}}^{\left(1\right)}$ equilibrium, which results in the positive equilibrium coming into the interior of the admissible region $\Omega$. After that, a saddle-node of limit cycles bifurcation occurs, which gives birth to a stable and unstable periodic orbits. The $\prescript{}{\left(111\right)}{\mathcal{E}}$ equilibrium (unstable), undergoes a Hopf bifurcation, and as a consequence it becomes asymptotically stable, and the stable periodic orbits disappears. Since these bifurcations occur for a narrow range of $\alpha$, a close-up is presented in Figure \ref{fig:bif_diag_alph_x0_zoom}. Stable periodic orbit represents a case in which all three populations oscillate indefinitely, and hence the production of methane fluctuates. As already mentioned in section \ref{Hopf_bifurcation}, this situation is not a desirable one, because it might result it decreased productivity of the biogas plant. The unstable periodic orbits acts as a separatrix, giving the border of the basin of attraction of two asymptotically stable equilibria in the case of bistability. 

Since by the conservation principles \eqref{eq:cons_laws}, $s_0=u_f-x_0$, the bifurcation diagram in $\alpha$ with $s_0$ on the $y$-axis is similar to the one presented in Figure \ref{fig:bif_diag_alph_x0}. The amount of chlorophenol in the system is inversely proportional to the concentration of the phenol degrader. As the dilution rate $\alpha$ decreases, concentration of the chlorophenol degrader in the interior equilibrium increases, and as a consequence, the amount of chlorophenol decreases. It thus suggests, that operating on lower dilution rates results in the most desirable dynamics, i.e., an asymptotically stable interior equilibrium and fast chlorophenol removal.

To extend the previous analysis, we now fix the following parameters
\begin{equation}
    u_g = 0, \quad u_h = 0.1,
\end{equation}
and plot a two-parameter bifurcation diagram of system \eqref{eq:system_reduced}, choosing $\alpha$ and $u_f$ as the bifurcation parameters. Each region of the diagram is labeled and the corresponding dynamics are represented schematically in figures around it. Black dashed curve corresponds to saddle-node of equilibria bifurcation (LP), black solid curve represents saddle-node of limit cycles bifurcation (SNLC), black dotted curve denotes Hopf bifurcation (HB), and grey solid curves represent transcritical bifurcations (BP). We also denote the predicted heteroclinic bifurction by a grey dashed curve which lies very close to the Hopf curve.



\begin{figure}[H]
\centering
    \includegraphics[angle=90,trim={1cm 1cm 3cm 1cm},clip]{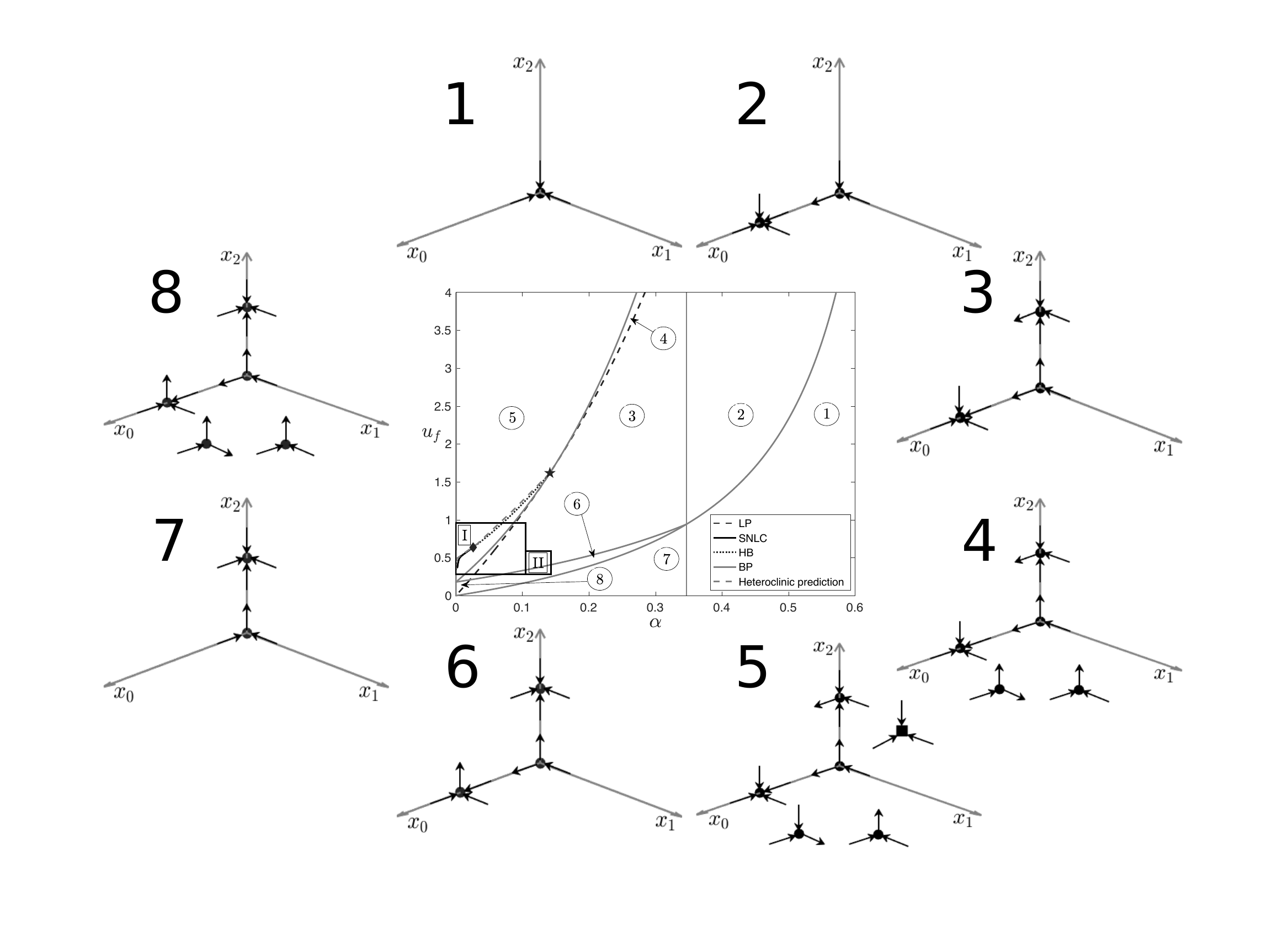}
    \caption{Two-parameter bifurcation diagram of system \eqref{eq:system_reduced} with $u_g=0$ and $u_h=0.1$.}
    \label{fig:2par_bif_diag_full}
\end{figure}

\begin{figure}[H]
    \centering
    {\includegraphics[angle=90, trim={2cm 1cm 2cm 1cm},clip]{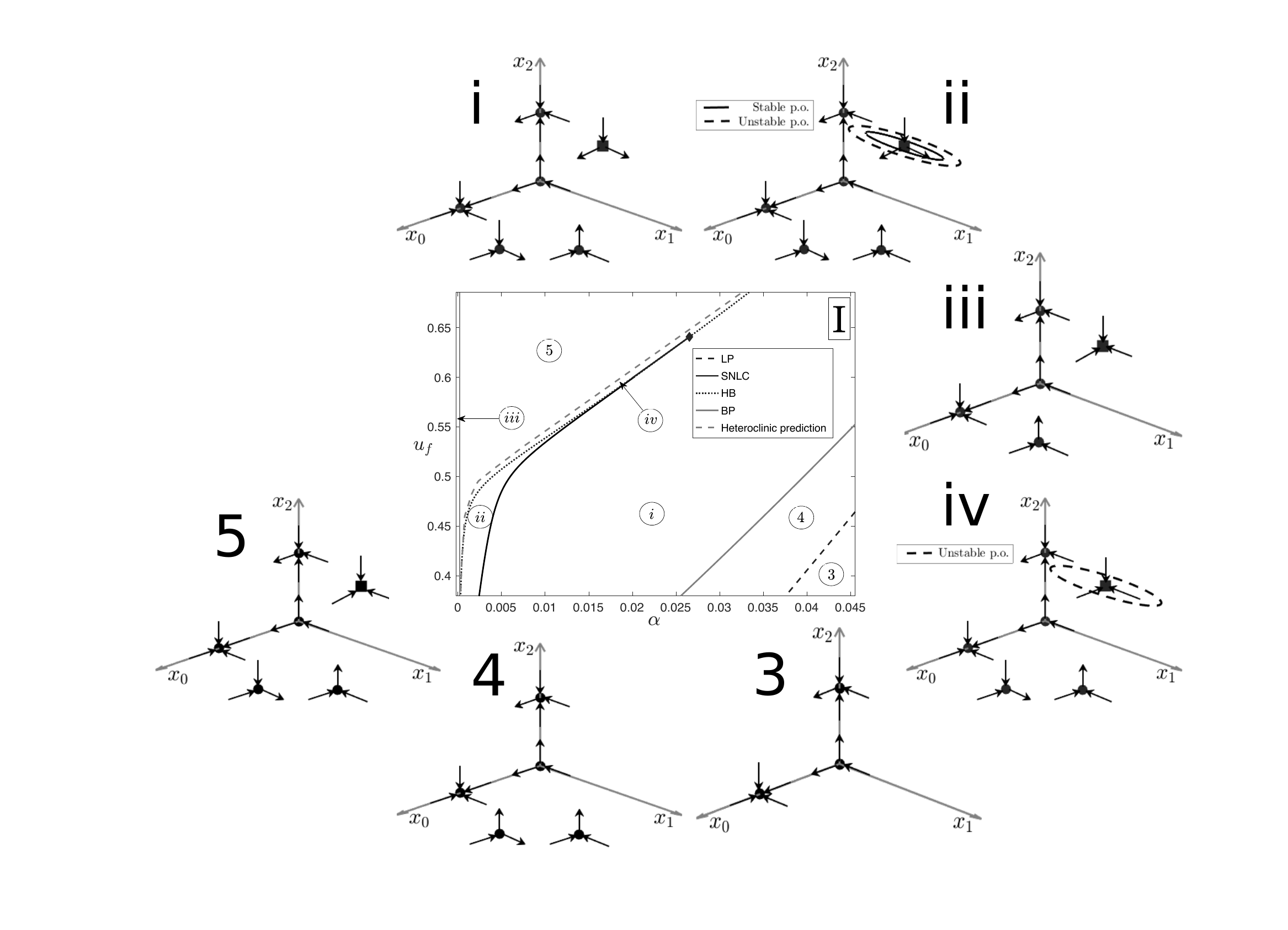}}
    \caption{Close-up on region I of Figure \ref{fig:2par_bif_diag_full}. The SNLC curve intersects the HB curve at Bautin bifurcation. This results in the change of criticality of the Hopf bifurcation from supercritical (on the left) to subcritical (on the right).}
    \label{fig:2par_bif_diag_zoom_I}
\end{figure}

\begin{figure}[H]
    {\includegraphics[angle=90, trim={2cm 2cm 1cm 0cm},clip]{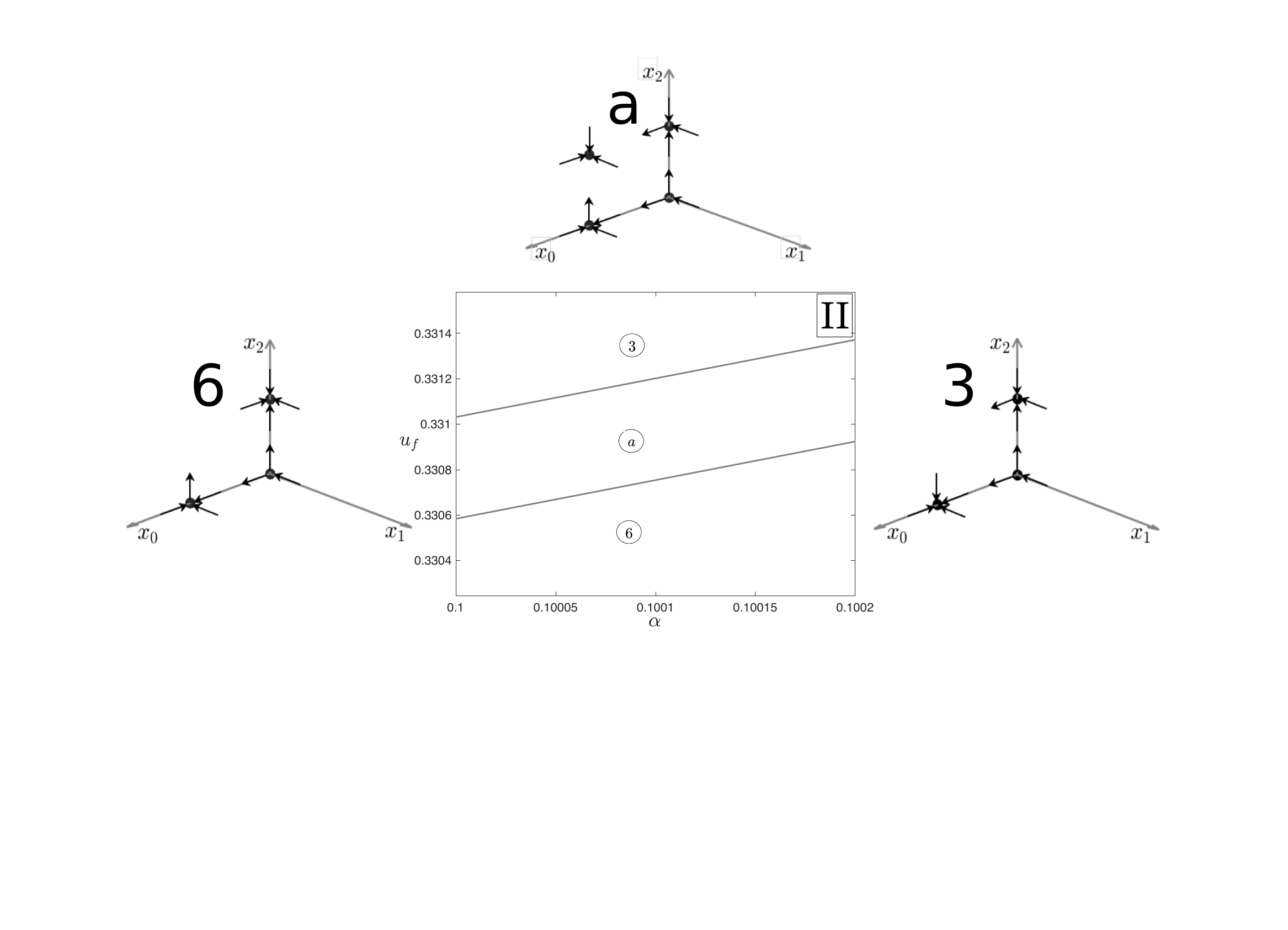}}
    \caption{Close-up on region II of Figure \ref{fig:2par_bif_diag_full}.}
    \label{fig:2par_bif_diag_zoom_II}
\end{figure}

We can see that varying two parameters at the same time can lead to a much more complicated dynamics than in the case of one-parameter bifurcations. There is a generalized Hopf bifurcation, at the point at which the Hopf curve intersects the saddle-node of limit cycles curve. This is the point where the criticality of the Hopf bifurcation changes from supercritical to subcritical, looking from left to right. The unstable periodic orbit disappears through a heteroclinic bifurcation. There are two heteroclinic orbits that form a cycle that joins the two equilibria in the $x_0x_1$ face, then passes into the interior, and then goes back to the boundary in the $x_0x_1$ face. The point at which the Hopf, homoclinic, and saddle-node of limits cycles curves intersect, represents the Bogdanov-Takens bifurcation.

From the biological viewpoint, the most interesting dynamics is occurs in regions $5$ and $iii$. There, the interior equilibrium is asymptotically stable. In the case of region $5$ we also observe bistability with the $\prescript{}{\left(100\right)}{\mathcal{E}}$ equilibrium. In region $iii$, there is uniform persistence, and thus an interior compact attractor is present. As was previously anticipated by the analysis of the one-parameter bifurcation diagram, operating at low dilution rates is the most desirable approach. If $\alpha$ is small enough, it is possible to remain in region $iii$, even for high inflow rate $u_f$.

\section{Conclusions}

In this work we have generalised the approach presented in \cite{sari2017} by including multiple substrate inflow into the chemostat, while maintaining generality (in most cases) with respect to the exact form of the growth functions. We observed that allowing the inflow of multiple substrates resulted in much more complex dynamics of the system. For example, eight steady states are possible. Previously, the theoretical results were limited to existence and uniqueness of up to three equilibria (when chlorophenol was the only input substrate), and to numerical evidence suggesting that the model should be subjected to a more detailed analysis. We also observed that external addition of substrates can result in bistability - two equilibria can simultaneously be asymptotically stable. As well, there can be an orbitally asymptotically stable periodic orbit  with all of the populations surviving and an asymptotically stable equilibrium with only chlorophenol degrader population surviving. 

We have also confirmed the findings of the previous analysis in \cite{sari2017}, where numerical evidence of the occurrence of a supercritical Hopf bifurcation was given. Theoretical conditions for the existence of a Hopf bifurcation were provided in the case of specific forms of the growth functions. Varying any one of the three parameters:  chlorophenol, phenol, and hydrogen inflow rates, was shown to result in a Hopf bifurcation. Theoretical results for varying the dilution rate as the bifurcation parameter has been left for future work. However, we have observed numerically, that varying this parameter can result in a Hopf bifurcation and a saddle-node of limit cycles bifurcation. Our numerical investigations also showed that increasing the inflow rate of the substrates has a stabilizing effect on the entire system. From a biological engineering point of view, i.e. a bioreactor treating a monochlorophenol rich waste stream, instability would typically be undesirable in terms of process performance. Therefore, identification of control strategies to avoid periodic behaviour is an important output of this work.

Another result, particularly important for engineering applications, concerns the persistence of the system for a range of parameter sets. Knowing when the microbial populations survive is again crucial from a process control perspective, and it is one of the main theoretical results of this work. We have proven that in two configurations of equilibria (in both cases all the boundary equilibria are saddle points) we observe not only persistence, but also uniform persistence, a much stronger result. These situations occur when there is an inflow of all three substrates, but also when phenol addition is not considered (i.e., when $u_g=0$).

Although we now know much more about the dynamics of the system, it is not fully understood. This follows from the numerical results provided by the two-parameter bifurcation diagrams. The analyses reveals that varying the dilution rate and the chlorophenol inflow simultaneously, can lead to a Bogdanov-Takens, or Bautin (generalized Hopf) bifurcations. Also, for the cases of bistability, where both a boundary and the interior equilibrium are asymptotically stable, it is of great importance to empiricists to have an estimation of the basins of attraction of these equilibria. This result is usually difficult to obtain theoretically, however numerical estimations are possible. Another factor that is of interest would be the inclusion of stochasticity in the model. In practice, even if the interior equilibrium is globally asymptotically stable, one of the microorganisms may become extinct. This might occur when a population is very small, and the stochastic noise effects result in the solution curve reaching one of the invariant faces of the admissible region.

There has been resistance to the idea that simplified models, of the type described here, are too remote from the systems they represent to be of worth to practitioners. Without experimental results to compare against model predictions, this case becomes stronger. However, we can look to emerging disciplines such as synthetic biology to help bridge the theoretical and the applied \cite{elkaroui19}. Recent studies have shown that synthetically derived anaerobic communities are able to confirm model predictions and provide insight into the ecology and dynamics of microbial communities that are relevant in practice \cite{delattre19}. We believe this work provides a basis by which experimental studies describing a chlorophenol mineralising food-web could be undertaken.

\section{Acknowledgments}

M.J.W. acknowledges the support from the European Union's Horizon 2020 research and innovation programme under the Marie Skłodowska-Curie Grant Agreement No. 702408 (DRAMATIC).

\bibliographystyle{siam}
\bibliography{references}

\end{document}